\documentclass[12pt,a4paper]{article}
\usepackage{amsfonts,amsthm,amsmath,amssymb,amscd}
\usepackage{fancybox}
\usepackage[latin1]{inputenc}
\usepackage[usenames]{color}
\usepackage{graphicx}
\usepackage{enumerate}
\usepackage{float}
\usepackage{fullpage}
\usepackage{epsfig,amssymb} 
\usepackage{subcaption}

\def\r{\mathbb R}

\def\h{\mathbb H}

\newtheorem{theorem}{Theorem}[section]
\newtheorem{corollary}[theorem]{Corollary}
\newtheorem{lemma}[theorem]{Lemma}
\newtheorem{proposition}[theorem]{Proposition}
\newtheorem{definition}[theorem]{Definition}
\newtheorem{conjecture}[theorem]{Conjecture}
\newtheorem{rem}[theorem]{Remark}

\def \r{\mbox{${\mathbb R}$}}
\def \h{\mbox{${\mathbb H}$}}

\begin{document}
	
\title{\textbf{Axially Symmetric Helfrich Spheres}}

\author{R. L\'opez, B. Palmer and A. P\'ampano}
\date{\today}

\maketitle

\begin{abstract}
Smooth axially symmetric Helfrich topological spheres are either round or else they must satisfy a second order equation known as the reduced membrane equation \cite{PP2}. In this paper, we show that, conversely, axially symmetric closed genus zero solutions of the reduced membrane equation which, in addition, satisfy a rescaling condition are axially symmetric Helfrich spheres. We also exploit this characterization to geometrically describe these surfaces and present convincing evidence that they are symmetric with respect to a suitable plane orthogonal to the axis of rotation and that they belong to a particular infinite discrete family of surfaces.\\

\end{abstract}

\begin{figure}[h!]
	\centering
	\includegraphics[height=6cm]{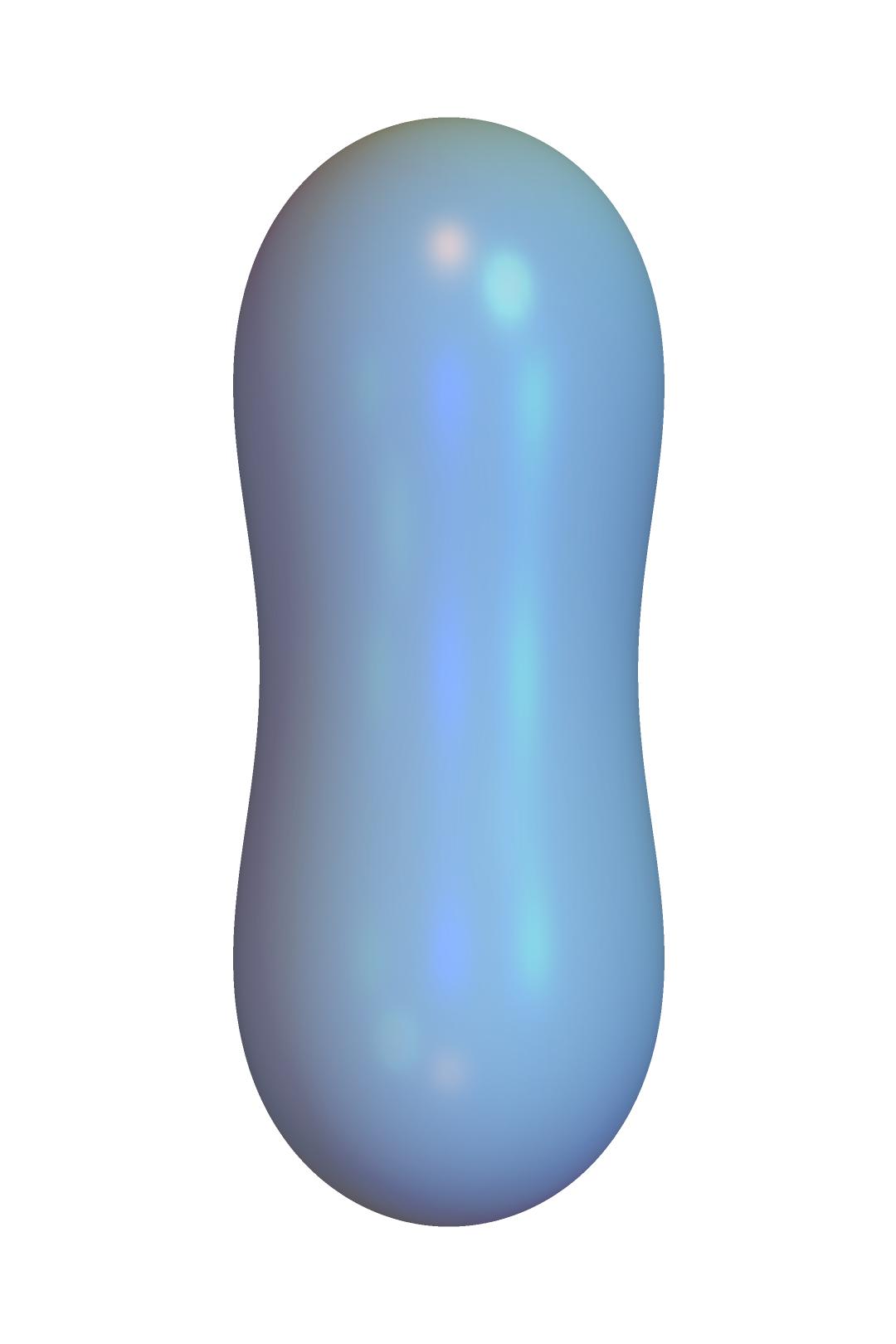}
\end{figure}


\section{Introduction}

The main constituents of cellular membranes in most living organisms are lipid bilayers formed from a double layer of phospholipids, which, in general, have distinct compositions and surface tensions. These bilipid membranes are usually modeled as mathematical surfaces since they are very thin compared to the size of the observed cells. In addition, for physical reasons related to the principle of non-interpenetration of matter, these surfaces are usually assumed to have no self-intersections. 

Historically, there has been a great interest in explaining the morphology of these membranes. Since membranes spontaneously close generating vesicles, it was recognized early on that their surface energy was not that of surface tension. In 1973, by comparing the similarities between a bilipid membrane and a thin layer of nematic liquid crystal, with the membrane's surface normal playing the role of the director field of the liquid crystal, the German physicist W. Helfrich \cite{H} established the surface energy of the bilipid membrane to be
$$\mathcal{H}_{a,c_o,b}[\Sigma]:=\int_\Sigma\left(a(H+c_o)^2+bK\right)d\Sigma\,,$$
where $H$ and $K$ denote, respectively, the mean and Gaussian curvature of the membrane surface $\Sigma$. This  is a type of elastic energy with three parameters that depend on the composition of the membrane itself: $a>0$ is the bending rigidity modulus, $b$ is the saddle-splay modulus and $c_o$ is the spontaneous curvature (this definition for $c_o$ differs from the classical one by the sign and a coefficient two). Assuming the membrane is homogeneous, these physical parameters are constant.

Bilipid membranes are composed of bipolar phospholipid molecules having a hydrophilic head and a hydrophobic tail. When a sufficiently high density of these molecules is reached in an aqueous solution, the molecules self-assemble to hide the tails from the surrounding solvent, preventing an energetically unfavorable scenario. This often results in a closed bilayer with the heads of the molecules pointing outwards (see \cite{BMF} and references therein). 

For closed surfaces $\Sigma$, the energy $\mathcal{H}_{a,c_o,b}$ is variationally equivalent to
$$\mathcal{H}[\Sigma]:=\int_\Sigma (H+c_o)^2\,d\Sigma\,,$$
since the integral of the Gaussian curvature can be reduced to a topological invariant by the Gauss-Bonnet theorem. The functional $\mathcal{H}$ is called the \emph{Helfrich energy} and its critical points are referred to as \emph{Helfrich surfaces}. It turns out that the only physical parameter playing a role in the determination of the shape of these surfaces is the spontaneous curvature $c_o$. This parameter originates primarily from the asymmetry between the two layers of the membrane, although it may also arise from differences in the chemical properties of the fluids on both sides of the lipid bilayer \cite{DSL,S}.

The Euler-Lagrange equation associated with the Helfrich energy $\mathcal{H}$ is the fourth order nonlinear elliptic partial differential equation
\begin{equation}\label{EulerLagrange}
	\Delta H+2(H+c_o)\left(H(H-c_o)-K\right)=0\,,
\end{equation}
where $\Delta$ denotes the Laplace-Beltrami operator. Observe that constant mean curvature surfaces with $H= -c_o$ trivially satisfy the above equation. If the constant mean curvature surface is a topological sphere, Hopf proved that it must be round \cite{Ho}. In addition, in the case of closed (of arbitrary genus) and embedded constant mean curvature surfaces, the classical Alexandrov's theorem yields the same conclusion \cite{A}. Hence, the main interest is to study  Helfrich surfaces with non-constant mean curvature. 

Other than round spheres with $H= -c_o$, few explicit examples of closed and embedded Helfrich surfaces are known in the literature. Arguably, the most common non-trivial examples are the \emph{circular biconcave discoids} \cite{NOO}, which are a well-known family of axially symmetric solutions of \eqref{EulerLagrange}. Unfortunately, these surfaces are singular \cite{Tu}. Indeed, they fail to be $\mathcal{C}^2$ at their two `poles' and also fail to be critical points of $\mathcal{H}$ (see Proposition \ref{discoids} below). Recently, there have been several works studying closed Helfrich surfaces, but these are often considered in combination with some constraints, such as, fixed area and/or fixed enclosed volume \cite{BWW,CV,MS} as well as fixed isoperimetric ratio \cite{KM}. Moreover, numerical approaches have also been recently published (see, for instance, \cite{CYBKYZ}). Nonetheless, solutions for the unconstrained case, that is, closed critical points for $\mathcal{H}$, give a rich family of surfaces which has yet to be fully understood. The lack of examples is mainly due to the fact that the Euler-Lagrange equation \eqref{EulerLagrange} is of order four and when $c_o\neq 0$, contrary to the Willmore case ($c_o=0$), the functional is not conformally invariant, which makes its analysis rather complicated. Similarly, since there are no restrictions imposed to the variations, the direct method of the calculus of variations does not directly apply. 

The second and third authors recently introduced in \cite{PP2} the second order equation
\begin{equation}\label{RME}
	H+c_o=\frac{\nu_3}{A-z}\,,
\end{equation}
and showed that, with enough regularity, it is a sufficient condition for \eqref{EulerLagrange} to hold. For this reason, this equation is referred to as the \emph{reduced membrane equation}. Here, $z$ is a coordinate of the surface, that as customary will be taken to be the vertical one; $\nu_3$ is the vertical component of the surface normal $\nu$; and, $A$ is a real constant, which can be assumed to be zero after a suitable translation of the vertical coordinate, if necessary. The reduced membrane equation \eqref{RME} is a necessary condition for the Euler-Lagrange equation \eqref{EulerLagrange} to hold if, for instance, the surface is assumed to be an axially symmetric topological sphere which is sufficiently regular so as to be a weak solution of \eqref{EulerLagrange}. 

In this paper we begin by showing in Theorem \ref{z=0} that $\mathcal{C}^3$ closed surfaces satisfying \eqref{RME} are Helfrich surfaces which, in addition, intersect the horizontal plane $\{z=0\}$ orthogonally at closed geodesics. The set of points of the surface where $z=0$ holds is a singularity of the equation \eqref{RME}. Thus, in order to obtain closed genus zero solutions it is necessary to glue together  two disc type surfaces along the points where $z=0$ holds. Even though there exist non-axially symmetric surfaces satisfying the reduced membrane equation as well \cite{PP3}, we focus on the axially symmetric case. In Section 3 we describe geometrically the generating curves of axially symmetric surfaces satisfying \eqref{RME} which meet the axis of rotation. In particular, we prove that infinitely many of them intersect the axis $z=0$ orthogonally (Theorems \ref{tc1} and \ref{TypeII}). Consequently, it is possible to generate closed surfaces satisfying \eqref{RME} by gluing together suitable axially symmetric disc type surfaces along their geodesic boundary circle located at $z=0$. Additionally, we then prove that the closed surfaces obtained this way are real analytic everywhere except perhaps, at $z=0$ where these surfaces have $\mathcal{C}^2$ regularity (Theorem \ref{regularity}). However, if $c_o\neq 0$, they may not have $\mathcal{C}^3$ regularity at $z=0$. It turns out that $\mathcal{C}^3$ regularity is a sufficient and necessary condition to obtain critical points of the Helfrich energy $\mathcal{H}$ (Lemma \ref{2}). When $c_o\neq 0$, this $\mathcal{C}^3$ regularity condition is equivalent to the mean value of $H+c_o$ on the surface being identically zero; an integral condition arising from deformations of the surface through rescalings. Therefore, we conclude in Theorem \ref{rescaling} with a characterization of axially symmetric Helfrich topological spheres (other than round ones) as solutions of the reduced membrane equation \eqref{RME} which satisfy the additional rescaling condition 
$$\int_\Sigma (H+c_o)\,d\Sigma=0\,.$$
In particular, the above rescaling condition is automatically satisfied if the mean value of $H+c_o$ is zero on both axially symmetric disc type surfaces glued together to obtain the topological sphere. We show in Proposition \ref{symmetry} that these two conditions imply that the closed surface is symmetric with respect to the plane $\{z=0\}$. Exploiting the characterization of Theorem \ref{rescaling} and the geometric property of Proposition \ref{symmetry}, in Section 5 we present several non-trivial examples of axially symmetric Helfrich surfaces which are symmetric with respect to the plane $\{z=0\}$ (illustrations of these surfaces can be found in Figure \ref{Helfrich} and in the last few pages of this paper.). Finally, we conjecture that, in fact, axially symmetric Helfrich topological spheres must be symmetric with respect to the plane $\{z=0\}$. Our conjecture is supported by convincing numerical evidence and its validity would show that our family of examples includes all axially symmetric Helfrich topological spheres with non-constant mean curvature, which in turn would lead to the complete classification of such surfaces.

\section{Reduced Membrane Equation}\label{s2}

Let $\r^3$ denote the three dimensional Euclidean space whose canonical coordinates are represented by $(x,y,z)$ and let $E_i$, $i=1,2,3$, be the constant unit vector fields in the direction of the coordinate axes.

Let $\Sigma$ be a closed, connected, and oriented surface. If, in addition, $\Sigma$ has genus zero we refer to it as a \emph{topological sphere}. Consider a $\mathcal{C}^2$ immersion $X:\Sigma\longrightarrow\mathbb{R}^3$ of the surface $\Sigma$ in $\mathbb{R}^3$ and denote by $\nu:\Sigma\longrightarrow\mathbb{S}^2$ a choice of unit normal vector field. 

\begin{rem}\label{embedding} Although most of our results are valid only assuming that $X:\Sigma\longrightarrow\r^3$ is an immersion, from now on we will restrict our attention to embedded surfaces since they are physically the most relevant ones. In this case, the normal field $\nu:\Sigma\longrightarrow\mathbb{S}^2$ is chosen to be the outward pointing one. 
\end{rem}

For the immersion $X:\Sigma\longrightarrow\r^3$ the \emph{Helfrich energy} \cite{H} is given by
\begin{equation}\label{energy}
	\mathcal{H}[\Sigma]:=\int_\Sigma (H+c_o)^2\,d\Sigma\,,
\end{equation}
where $H$ is the mean curvature of $X$ and $c_o\in\r$ is a constant. The specific case $c_o=0$ is the classical Willmore energy \cite{W}. It is worth noticing that, for $c_o\neq 0$, the Helfrich energy is not conformally invariant. Moreover, it is not invariant by a change of orientation either, since that will transform the spontaneous curvature as $c_o\mapsto -c_o$.

Let $\psi\in\mathcal{C}_o^\infty(\Sigma)$ be a compactly supported smooth function defined on $\Sigma$ and consider the variation vector field $\delta X=\psi\nu$, which is the linear term of a deformation $X_\epsilon:=X+\epsilon\delta X+\mathcal{O}(\epsilon^2)$. For arbitrary variations $X_\epsilon$ of the immersion $X:\Sigma\longrightarrow\r^3$ with sufficient regularity, we have the following variation formula (for details see Appendix A of \cite{PP1})
\begin{equation}\label{fvf}
	\delta\mathcal{H}[\Sigma]=\int_\Sigma\left(-\nabla H\cdot\nabla\psi+2(H+c_o)\left(H(H-c_o)-K\right)\psi\right)d\Sigma\,,
\end{equation}
where $K$ denotes the Gaussian curvature of the immersion.

\begin{definition}
Critical points (or, equilibria) for the Helfrich energy $\mathcal{H}$ are called Helfrich surfaces.
\end{definition}

For immersions $X:\Sigma\longrightarrow\r^3$ of class $\mathcal{C}^4$, we deduce integrating by parts the first term in \eqref{fvf} and applying the Fundamental Lemma of the Calculus of Variations that the Euler-Lagrange equation associated with the Helfrich energy $\mathcal{H}$ is the fourth order nonlinear elliptic partial differential equation
\begin{equation}\label{EL}
	\Delta H+2\left(H+c_o\right)\left(H(H-c_o)-K\right)=0\,,
\end{equation}
where $\Delta$ denotes the Laplace-Beltrami operator. In other words, Helfrich surfaces of class $\mathcal{C}^4$ are characterized by the classical solutions of \eqref{EL}.

\begin{rem}\label{elliptic} 
	By elliptic regularity (see Theorem 6.6.1 of \cite{M}), it follows that immersions $X:\Sigma\longrightarrow\r^3$ of class $\mathcal{C}^3$ that are Helfrich surfaces are real analytic. Thus, classical solutions of \eqref{EL}.
\end{rem}

Any surface with constant mean curvature $H= -c_o$ trivially satisfies \eqref{EL} and, hence, it is a Helfrich surface. Employing the moving plane method, Alexandrov showed that the only closed and embedded constant mean curvature surfaces in $\r^3$ are round spheres \cite{A}. Moreover, Hopf proved that topological spheres immersed in $\mathbb{R}^3$ with constant mean curvature must be round as well \cite{Ho}. The spheres satisfying $H= -c_o$ are not only Helfrich surfaces, but are also obviously the absolute minimizers of the Helfrich energy $\mathcal{H}$. 

In order to find non-trivial solutions of \eqref{EL}, we consider the second order equation
\begin{equation}\label{RME0}
	H+c_o=-\frac{\nu_3}{z}\,,
\end{equation}
which is known as the \emph{reduced membrane equation} \cite{PP2}. Observe that $\nu_3=\nu\cdot E_3$ is the vertical component of the unit normal $\nu$. We point out here that, when $c_o\neq 0$, this equation is not invariant by a change of orientation. Indeed, changing the orientation takes $c_o\mapsto -c_o$, as is the case of the Helfrich energy itself. However, a useful property of this equation is that it remains invariant after horizontal translations, rotations around the $z$-axis, reflections across vertical planes as well as across the horizontal plane $\{z=0\}$.

\begin{proposition}\label{reflection}
	Let $X:\Sigma\longrightarrow\r^3$ be an immersion satisfying \eqref{RME0}. Then, horizontal translations of $X$, rotations of $X$ about the $z$-axis, reflections of $X$ across any vertical plane, and the reflection of $X$ across the horizontal plane $\{z=0\}$ also satisfy \eqref{RME0}.
\end{proposition}
\begin{proof} Assume that the immersion $X:\Sigma\longrightarrow\r^3$ satisfies \eqref{RME0} for a fixed constant $c_o\in\r$. The mean curvature $H$ stays invariant after any of the mentioned isometries. Moreover, it is clear that horizontal translations, rotations around the $z$-axis and reflections across vertical planes also preserve both $z$ and $\nu_3$. Hence, \eqref{RME0} remains invariant. We now consider the reflection of $X$ across the plane $\{z=0\}$. Under this transformation, $H\mapsto H$, $\nu_3\mapsto -\nu_3$ and $z\mapsto -z$. Thus, equation \eqref{RME0} is satisfied for the same $c_o\in\r$.
\end{proof}

\begin{rem}\label{cylinder} If $c_o=0$, round spheres of arbitrary radii satisfy the reduced membrane equation \eqref{RME0}. If $c_o\neq 0$, the right circular cylinder of radius $1/(2\lvert c_o\rvert)$ satisfies \eqref{RME0} (for a suitable orientation depending on the sign of $c_o$). Although the right circular cylinder is not closed, it is an axially symmetric example of Helfrich surface that will arise later on. Other examples of non-closed Helfrich surfaces can be obtained by constructing a right cylinder $X(s,t)=\gamma(s)+t E_3$, where $\gamma(s)\subset\{z=0\}$ is any elastic curve circular at rest.
\end{rem}

The reduced membrane equation \eqref{RME0} is the Euler-Lagrange equation for the action
$$\mathcal{G}[\Sigma]:=\int_\Sigma \frac{1}{z^2}\,d\Sigma-2c_o\int_\Omega \frac{1}{z^2}\,dV\,,$$
where $\Omega\subset\r^3$ is the volume enclosed by the surface $\Sigma$. For any part of the surface which lies in a half-space $\{z>0\}$ or $\{z<0\}$ of $\r^3$, this functional can be expressed as
$$\mathcal{G}[\Sigma]=\widetilde{\mathcal{A}}[\Sigma]-2c_o\int_{\widetilde{\Omega}}\lvert z\rvert\,d\widetilde{V}=\widetilde{\mathcal{A}}[\Sigma]-2c_o\,\widetilde{\mathcal{U}}[\Sigma]\,,$$
where $\widetilde{\mathcal{A}}$ denotes the area of $\Sigma$ regarded, after a conformal change of metric, as a surface in the hyperbolic space $\h^3$, $\widetilde{\Omega}$ denotes the hyperbolic volume enclosed by $\Sigma$, and $\widetilde{\mathcal{U}}$ is the gravitational potential energy of $\Sigma$, considered as a surface in $\h^3$. Consequently, the solutions of \eqref{RME0} can be viewed as minimal surfaces if $c_o=0$ or, if $c_o\ne 0$, as capillary surfaces with  a constant vertical gravitational force  in the hyperbolic space. The axially symmetric case gives rise to  hyperbolic versions of sessile or pendant drops. 
 
In Proposition 4.1 of \cite{PP2}, it was shown that if a \emph{sufficiently regular} immersion $X:\Sigma\longrightarrow\r^3$ has mean curvature satisfying \eqref{RME0}, then \eqref{EL} also holds and, consequently, the immersion is critical for the Helfrich energy $\mathcal{H}$. In the following theorem we show that $\mathcal{C}^3$ regularity is enough to obtain that the surface is a Helfrich surface and give necessary conditions for the existence of closed surfaces satisfying the reduced membrane equation. Indeed, we will see that these surfaces cannot be entirely contained in either of the half-spaces $\{z>0\}$ or $\{z<0\}$. (Although
throughout this paper we are mainly focusing on closed genus zero surfaces, this result is valid for general closed surfaces, regardless of their genera.)

\begin{theorem}\label{z=0}
	Let $\Sigma$ be a closed surface and $X:\Sigma\longrightarrow\r^3$ a $\mathcal{C}^3$ immersion satisfying the reduced membrane equation \eqref{RME0}. Then, $X(\Sigma)$ is a real analytic Helfrich surface which intersects the plane $\{z=0\}$ orthogonally at closed geodesics. Moreover, if $c_o\neq 0$, the rescaling condition
	$$\int_\Sigma\left(H+c_o\right)d\Sigma=0\,,$$
	holds.
\end{theorem}

\begin{rem} The previous result also holds for the Willmore case ($c_o=0$). If $\Sigma$ is a topological sphere we can conclude from Theorem 1.1 of \cite{P} that the surface is a round sphere. However, if no restriction on the genus of $\Sigma$ is made, then there exist examples of nontotally umbilical Willmore surfaces with non-constant mean curvature \cite{BB}.
\end{rem}

The proof of Theorem \ref{z=0} will follow from two lemmas. For these lemmas we let $\Sigma$ be a closed surface of arbitrary genus and $X:\Sigma\longrightarrow\mathbb{R}^3$ a $\mathcal{C}^3$ immersion satisfying \eqref{RME0}.

We first show that the immersion $X:\Sigma\longrightarrow\mathbb{R}^3$ is critical for the Helfrich energy $\mathcal{H}$.

\begin{lemma}\label{n1} A $\mathcal{C}^3$ immersion $X:\Sigma\longrightarrow\mathbb{R}^3$ satisfying \eqref{RME0} is critical for the Helfrich energy $\mathcal{H}$. In particular, $X$ is real analytic.
\end{lemma}
\begin{proof} We will show that the $\mathcal{C}^3$ immersion $X:\Sigma\longrightarrow\mathbb{R}^3$ is a weak solution of \eqref{EL}, by checking that \eqref{fvf} vanishes. For that, we use \eqref{RME0} to compute
$$\int_\Sigma -\nabla H\cdot\nabla\psi\,d\Sigma=\int_\Sigma \nabla\left(\frac{\nu_3}{z}\right)\cdot\nabla\psi\,d\Sigma=\int_\Sigma\left(\frac{\nabla\nu_3\cdot \nabla\psi}{z}-\frac{\nu_3 \nabla z\cdot\nabla\psi}{z^2}\right)d\Sigma\,,$$
where $\psi\in\mathcal{C}_o^\infty(\Sigma)$ is an arbitrary compactly supported smooth function defined on $\Sigma$.

Now, both terms above have sufficient regularity to be integrated by parts. Thus,
\begin{eqnarray}\label{*}
	\int_\Sigma -\nabla H\cdot\nabla\psi\,d\Sigma&=&\int_\Sigma\left(\nabla\cdot\left(\frac{\nu_3\nabla z}{z^2}\right)-\nabla\cdot\left(\frac{\nabla \nu_3}{z}\right)\right)\psi\,d\Sigma\nonumber\\
	&=&\int_\Sigma \left(2\frac{\nabla\nu_3\cdot\nabla z}{z^2}-2\frac{\nu_3\lVert\nabla z\rVert^2}{z^3}+\frac{\nu_3\Delta z}{z^2}-\frac{\Delta \nu_3}{z}\right)\psi\,d\Sigma\,.
\end{eqnarray}
Consider a variation field $\delta X=E_3=\nabla z+\nu_3\nu$. This generates a variation of the immersion $X:\Sigma\longrightarrow\mathbb{R}^3$ through translations and, hence, its mean curvature $H$ remains invariant. Therefore, we get (for details on the computation of $\delta H$, see Appendix A of \cite{PP1})
$$0=\delta H=\frac{1}{2}\Delta\nu_3+\frac{1}{2}\lVert d\nu\rVert^2\nu_3+\nabla H\cdot \nabla z\,.$$
We then deduce that
\begin{eqnarray}\label{nu3}
	\Delta\nu_3&=&-2\nabla H\cdot\nabla z-\lVert d\nu\rVert^2\nu_3=2\nabla\left(\frac{\nu_3}{z}\right)\cdot\nabla z-\left(4H^2-2K\right)\nu_3\nonumber\\
	&=&2\frac{\nabla\nu_3\cdot\nabla z}{z}-2\frac{\nu_3\lVert \nabla z\rVert^2}{z^2}-\left(4H^2-2K\right)\nu_3\,,
\end{eqnarray}
where we have used that $\lVert d\nu\rVert^2=4H^2-2K$ and \eqref{RME0}. Combining \eqref{*} and \eqref{nu3} we conclude with
\begin{eqnarray}\label{**}
	\int_\Sigma-\nabla H\cdot\nabla \psi\,d\Sigma&=&\int_\Sigma \left((4H^2-2K)\frac{\nu_3}{z}+\frac{\nu_3\Delta z}{z^2}\right)\psi\,d\Sigma\nonumber\\
	&=&\int_\Sigma -2(H+c_o)\left(2H^2-K-H(H+c_o)\right)\psi\,d\Sigma\,.
\end{eqnarray}
In the last equality we have used twice that \eqref{RME0} holds on $\Sigma$ and that the Laplacian, with respect to the metric induced on $\Sigma$, of the height function $z=X\cdot E_3$ is
\begin{equation}\label{Dz}
	\Delta z=\Delta \left(X\cdot E_3\right)=\Delta X\cdot E_3=2H\nu\cdot E_3=2H\nu_3\,.
\end{equation}
Consequently, substituting \eqref{**} in \eqref{fvf} we have 
$$\delta\mathcal{H}[\Sigma]=\int_\Sigma 2(H+c_o)\left(-2H^2+K+H(H+c_o)+H(H-c_o)-K\right)\psi\,d\Sigma=0\,.$$
Hence, $X:\Sigma\longrightarrow\mathbb{R}^3$ is a Helfrich surface. By elliptic regularity \cite{M} (c.f., Remark \ref{elliptic}), the immersion $X$ is not only of class $\mathcal{C}^3$ but also real analytic.
\end{proof}

We next show that $X(\Sigma)$ intersects the plane $\{z=0\}$ orthogonally. (The following lemma will necessitate the immersion $X:\Sigma\longrightarrow\mathbb{R}^3$ to be an embedding, c.f., Remark \ref{embedding}.)

\begin{lemma}\label{n2} The closed surface $X(\Sigma)$ intersects the plane $\{z=0\}$ orthogonally and the connected components of the intersection $X(\Sigma)\cap\{z=0\}$ are closed geodesics of $\Sigma$.
\end{lemma}
\begin{proof}
First we show that the surface intersects the plane $\{z=0\}$. To the contrary, suppose that $X(\Sigma)\cap\{z=0\}=\emptyset$. That is, the surface $X(\Sigma)$ is contained in its entirety in the upper (or, lower) half-space $\{z>0\}$ (respectively, $\{z<0\}$). From \eqref{Dz} and \eqref{RME0}, we compute 
$$\Delta z=2H\nu_3=2\left(-\frac{\nu_3}{z}-c_o\right)\nu_3=-2\frac{\nu_3^2}{z}-2c_o\nu_3\,.$$
Integrating this over the closed surface $\Sigma$, we get
\begin{equation}\label{idz}
	0=\int_\Sigma \Delta z\,d\Sigma=-2\int_\Sigma \frac{\nu_3^2}{z}\,d\Sigma-2c_o\int_\Sigma \nu_3\,d\Sigma\,.
\end{equation}
For the second integral on the right-hand side, we apply the divergence theorem. Denoting by $\Omega\subset\r^3$ the open domain representing the volume enclosed by $\Sigma$, we have
$$\int_\Sigma \nu_3\,d\Sigma=\int_\Sigma E_3\cdot\nu\,d\Sigma=\int_\Omega \overline{\nabla}\cdot E_3\,dV=0\,,$$
where $\overline{\nabla}\cdot$ denotes the divergence operator in $\mathbb{R}^3$. Consequently, \eqref{idz} reduces to 
$$0=-2\int_\Sigma \frac{\nu_3^2}{z}\,d\Sigma\,,$$
which is a contradiction because by assumption $z$ is positive (or, negative) everywhere on $\Sigma$ and $\nu_3$ is not identically zero for a closed surface.

Once we have showed that $X(\Sigma)\cap\{z=0\}\neq \emptyset$, we will check that this intersection is orthogonal. This is a consequence of the regularity. Since the immersion $X$ is $\mathcal{C}^3$ (indeed, we have proved in Lemma \ref{n1} that it is real analytic), its mean curvature $H$ is well defined everywhere, and so is the left-hand side of \eqref{RME0}. Hence, we deduce that, necessarily, $\nu_3=0$ at all the points where $z=0$, or otherwise we would reach a contradiction. That is, the intersection is orthogonal.

Finally, we show that as a consequence of this intersection being orthogonal, the connected components of $X(\Sigma)\cap\{z=0\}$ are geodesics of $\Sigma$. Denote by $C$ any of these connected components and assume it is parameterized by its arc length parameter $\sigma$. The vector field $T(\sigma):=C_\sigma(\sigma)$ denotes the unit tangent to $C$, which clearly lies in the plane $\{z=0\}$. Observe that the normal vector field $\nu$ along $C$ also lies in $\{z=0\}$ due to the orthogonality of the intersection. Hence, from the definition of the geodesic curvature $\kappa_g$ of $C$ (see, for instance, Section 2 of \cite{PP2}), we then compute
$$\kappa_g=T_\sigma\cdot n=0\,,$$
since the conormal $n:=T\times \nu$ to $C$ is orthogonal to the plane $\{z=0\}$ and $C$ is a planar curve. This means that $C$ is a geodesic of $\Sigma$.
\end{proof}

We can now proceed with the proof of Theorem \ref{z=0}.
\\

\noindent{\emph{Proof of Theorem \ref{z=0}.} Let $X:\Sigma\longrightarrow\mathbb{R}^3$ be a $\mathcal{C}^3$ immersion of a closed surface satisfying \eqref{RME0}. From Lemma \ref{n1}, $X$ is a Helfrich surface, while Lemma \ref{n2} shows that $X(\Sigma)$ intersects the plane $\{z=0\}$ orthogonally at closed geodesics.
	
It only remains to prove the rescaling condition. This equality arises from a rescaling argument of the Helfrich energy $\mathcal{H}$ (analogous to that of Proposition 2.1 of \cite{PP1}). For a rescaling $R\Sigma$, $R>0$, the Helfrich energy $\mathcal{H}$ is given by
$$\mathcal{H}[R\Sigma]=\int_\Sigma \left(\frac{H}{R}+c_o\right)^2R^2\,d\Sigma=\int_\Sigma\left(H+c_o R\right)^2\,d\Sigma\,.$$
Differentiating this expression with respect to $R$ at the critical immersion $X:\Sigma\longrightarrow\r^3$ characterized by $R=1$ then yields
$$0=2c_o\int_\Sigma\left(H+c_o\right)d\Sigma\,.$$
Since $c_o\neq 0$, we get the integral condition of the statement. This finishes the proof. \hfill$\square$
\\

To finish this section, we show that $\mathcal{C}^3$ closed solutions of \eqref{RME0} satisfying an extra condition at $\{z=0\}$ are axially symmetric having the $z$-axis as the axis of rotation.

\begin{proposition}\label{n4} Let $\Sigma$ be a closed surface and $X:\Sigma\longrightarrow\r^3$ a $\mathcal{C}^3$ immersion satisfying the reduced membrane equation \eqref{RME0}. If the mean curvature $H$ and its derivative in the conormal direction are constant along one connected component of $X(\Sigma)\cap\{z=0\}$, then the surface is axially symmetric and its axis of rotation is the $z$-axis.
\end{proposition}
\begin{proof} Denote by $C$ any connected component of $X(\Sigma)\cap\{z=0\}$ and by $\sigma$ the arc length parameter of $C$. Recall from Lemma \ref{n2} that $C$ is a geodesic of $\Sigma$. In addition, Lemma \ref{n2} also shows that the surface $X(\Sigma)$ intersects the plane $\{z=0\}$ orthogonally. It then follows from the classical Joachimsthal theorem that $X(\Sigma)\cap\{z=0\}$ is composed by lines of curvature of $\Sigma$, that is, $\tau_g= 0$ holds everywhere on $C$, where $\tau_g$ denotes the geodesic torsion of $C$.

Differentiating the reduced membrane equation \eqref{RME0}, we obtain
\begin{equation}\label{needed}
	\nabla\left(H+c_o\right)=\nabla\left(-\frac{\nu_3}{z}\right)=\frac{1}{z}\left(-\nabla\nu_3+\frac{\nu_3}{z}\nabla z\right)=\frac{1}{z}\left(-\nabla\nu_3-(H+c_o)\nabla z\right).
\end{equation}

Since the immersion $X:\Sigma\longrightarrow\mathbb{R}^3$ is of class $\mathcal{C}^3$, $\nabla(H+c_o)$ is well defined everywhere on $\Sigma$. In particular, along $C$, $z=0$ holds so it must follow that
$$-\nabla\nu_3-(H+c_o)\nabla z=\left(\kappa_1-H-c_o\right)\nabla z=\left(H-\kappa_n-c_o\right)\nabla z=0\,,$$
where $\kappa_1$ is the principal curvature such that $\kappa_1=2H-\kappa_n$ along $C$ because $\tau_g=0$ holds (here, $\kappa_n$ denotes the normal curvature of $C$). Observe that in the first equality above we have used $\nabla\nu_3=d\nu\left(\nabla z\right)=-\kappa_1\nabla z$. Therefore, $H-\kappa_n-c_o=0$ along $C$, because $\nabla z\neq 0$ due to the orthogonality of the intersection (c.f., Lemma \ref{n2}).

Since $H$ is constant along $C$, we deduce from $\kappa_n=H-c_o$ that the normal curvature of $C$ is also constant. Then, the curvature $\kappa$ of $C$ as a curve in $\{z=0\}\subset\mathbb{R}^3$ satisfies
$$\kappa^2=\kappa_g^2+\kappa_n^2=\kappa_n^2\,,$$
which is constant. Here, $\kappa_g$ denotes the geodesic curvature of $C$, which vanishes since $C$ is a geodesic of $\Sigma$ (Lemma \ref{n2}). Hence, since $C$ is closed and has constant curvature $\kappa$, it must be a circle.
	
Combining this information with $\partial_nH$ being constant on $C$ (recall that $n$ denotes the conormal to the circle $C$ and so $\partial_n$ represents the derivative in the conormal direction), we are in a condition to apply Lemma 4.1 of \cite{PP2}. Hence, concluding that the surface is axially symmetric. Observe that in Lemma 4.1 of \cite{PP2} the axis of rotation is a vertical line passing through the center of the circle $C$.
\end{proof} 

\section{Axially Symmetric Solutions}

Theorem \ref{z=0} (more precisely, Lemma \ref{n2}) shows that the singularity $z=0$ of \eqref{RME0} cannot be avoided if one aims to obtain closed Helfrich surfaces employing the second order reduction \eqref{RME0}. This complicates our task since many techniques coming from the theory of partial differential equations do not apply at this singularity. For instance, the maximum principle and elliptic regularity results cannot be employed directly to study closed solutions of \eqref{RME0}. In this section we will study the case of axially symmetric topological spheres. We will first describe their generating curves and then use those that cut the plane $\{z=0\}$ to construct closed axially symmetric surfaces satisfying \eqref{RME0} everywhere.

\subsection{Geometric Description of the Generating Curves}

Let $X:\Sigma\longrightarrow\r^3$ be an axially symmetric surface whose axis of rotation is the $z$-axis and assume that \eqref{RME0} holds on $\Sigma$. Denote by $\gamma(s)=(r(s),z(s))$ the generating curve of $\Sigma$ included in the half-plane $\{x=r\geq 0\,,\,y=0\}$ and suppose that $\gamma$ is parameterized by the arc length. We then define a function $\varphi(s)$ by  $r'(s)=\cos\varphi(s)$ and $z'(s)=\sin\varphi(s)$, where $\left(\,\right)'$ denotes the derivative with respect to the arc length parameter $s$. The function $\varphi(s)$ represents the angle between the positive part of the $r$-axis and the tangent vector to $\gamma(s)$. It follows from \eqref{RME0}, that   $\gamma(s)$ satisfies the following system of first order ordinary differential equations:
\begin{eqnarray}
	r'(s)&=&\cos\varphi(s)\,,\label{system1}\\
	z'(s)&=&\sin\varphi(s)\,,\label{system2}\\
	\varphi'(s)&=&-2\frac{\cos\varphi(s)}{z(s)}-\frac{\sin\varphi(s)}{r(s)}-2c_o\,.\label{system3}
\end{eqnarray}

\begin{rem}\label{+-} As noticed in \cite{PP2,PP3}, equation \eqref{system3} carries both the positive and negative signs in front of $c_o$, depending on whether the arc length parameter $s$ is measured from the north or south `poles'. However, up to the transformation $z\mapsto -z$, $s$ may be assumed to be measured from the north `pole' and this sign can be fixed to be negative.
\end{rem}

We will now impose the initial conditions at $s=0$. Since we are looking for solutions that cut the axis of rotation, i.e., the $z$-axis, we have $r(0)=0$. The initial height will be considered as a parameter, so we take $z(0)=z_0\neq 0$. Moreover, for regularity purposes the initial angle must satisfy $\varphi(0)=0$. Indeed, we next show that if a solution of \eqref{system1}-\eqref{system3} cuts the $z$-axis, this cut must be orthogonal.

\begin{proposition}\label{orthogonal}
	Let $\{r(s),z(s),\varphi(s)\}$ be a solution of \eqref{system1}-\eqref{system3} such that $r(0)=0$ with $r'(0)\neq 0$ and $z(0)=z_0\neq 0$. Then, $\varphi(0)=0$ holds. In other words, the curve $\gamma(s)=(r(s),z(s))$ intersects the $z$-axis orthogonally.
\end{proposition}
\begin{proof}
	Let $\gamma(s)=(r(s),z(s))$ be the curve constructed from a solution of \eqref{system1}-\eqref{system3} such that $r(0)=0$ and $z(0)=z_0\neq 0$. Since $r'(0)\neq 0$, near $r=0$, we can describe the curve $\gamma$ as a graph $z= z(r)$. Then, \eqref{system3} can be rewritten as
	\begin{equation*}\label{er}
		\left(\frac{r z_r}{2\sqrt{1+z_r^2}}\right)_r=-\frac{r}{z\sqrt{1+z_r^2}}-c_or\,.
	\end{equation*}
	Integrating from $r=0$ to $r$, we have
	$$\frac{z_r}{\sqrt{1+z_r^2}}=-\frac{2}{r}\int_0^r \frac{t}{z(t)\sqrt{1+z_r^2(t)}}\,dt-c_o r\,.$$
	Since $z(0)=z_0\neq 0$, for sufficiently small $\epsilon>0$, there exists a positive constant $\lambda>0$ (depending on $\epsilon$) such that $\lvert z(r)\rvert \geq \lambda$ for all $r\in[0,\epsilon]$. Thus,
	\begin{equation*}
			\left|\frac{z_r}{\sqrt{1+z_r^2}}\right|\leq\frac{2}{r}\int_0^r\frac{t}{\lvert z\rvert\sqrt{1+z_r^2}}\,dt+\lvert c_o\rvert r\leq \frac{2}{r \lambda}\int_0^r t\,dt+\lvert c_o\rvert r=\left(\frac{1}{\lambda}+\lvert c_o\rvert \right)r\,.
	\end{equation*}
	Letting $r\to 0$, we conclude that $z_r(0)=0$, as claimed.
\end{proof}

To sum up, the appropriate initial conditions for the system \eqref{system1}-\eqref{system3} are:
\begin{equation}\label{conditions}
	r(0)=0\,,\quad\quad\quad z(0)=z_0\neq 0\,,\quad\quad\quad\varphi(0)=0\,.
\end{equation}

\begin{rem}\label{circle} If $c_o=0$, the unique solution to the system \eqref{system1}-\eqref{system3} with initial conditions \eqref{conditions} is (a part of) a circle of radius $\lvert z_0\rvert$. A more general result was obtained using Bryant's quartic differential in Theorem 1.1 of \cite{P}.
\end{rem}

Next, we will show that the value $c_o\neq 0$ can be fixed, possibly changing the initial height.

\begin{proposition}\label{+}
	Let $\{r(s),z(s),\varphi(s)\}$ be a solution of \eqref{system1}-\eqref{system3} with initial conditions \eqref{conditions}. Then, for every $\lambda\neq 0$,
	$$\widetilde{r}(s)=\lambda r(s/\lambda)\,,\quad\quad\quad \widetilde{z}(s)=\lambda z(s/\lambda)\,,\quad\quad\quad \widetilde{\varphi}(s)=\varphi(s/\lambda)\,,$$
	is a solution of \eqref{system1}-\eqref{system3} changing $c_o$ by $c_o/\lambda$ and with initial conditions
	$$\widetilde{r}(0)=0\,,\quad\quad\quad \widetilde{z}(0)=\lambda z_0\neq 0\,,\quad\quad\quad \widetilde{\varphi}(0)=0\,.$$
\end{proposition}
\begin{proof} Assume that the functions $\{r(s),z(s),\varphi(s)\}$ satisfy the system \eqref{system1}-\eqref{system3}. Differentiating the new function $\widetilde{r}(s)$ with respect to $s$, we have
	$$\widetilde{r}\,'(s)=\lambda\frac{d}{ds}\left(r(s/\lambda)\right)=r'(s/\lambda)=\cos\varphi(s/\lambda)=\cos\widetilde{\varphi}(s)\,,$$
	where we have used that $r$ satisfies \eqref{system1} and the definition of $\widetilde{\varphi}$. Similarly, we check that $\widetilde{z}\,'(s)=\sin\widetilde{\varphi}(s)$, while for $\widetilde{\varphi}\,'(s)$ we have
	$$\widetilde{\varphi}\,'(s)=\frac{1}{\lambda}\left(-2\frac{\cos\widetilde{\varphi}(s)}{\widetilde{z}(s)/\lambda}-\frac{\sin\widetilde{\varphi}(s)}{\widetilde{r}(s)/\lambda}-2c_o\right)=-2\frac{\cos\widetilde{\varphi}(s)}{\widetilde{z}(s)}-\frac{\sin\widetilde{\varphi}(s)}{\widetilde{r}(s)}-2\frac{c_o}{\lambda}\,.$$
	Finally, evaluating at $s=0$, the new initial conditions follow immediately. 
\end{proof}

From now on, we will assume that the spontaneous curvature is a fixed positive constant, that is, $c_o>0$. This can be done without any loss of generality, possibly applying Proposition \ref{+} for $\lambda=-1$. Since our initial height is going to be a parameter, such a transformation will not eliminate any cases. Notice that from Proposition \ref{+}, not only the sign, but also the magnitude of $c_o$ may be fixed. However, due to the physical meaning of the spontaneous curvature, we will only restrict its sign while keeping $c_o>0$ everywhere to highlight the role it plays on the shape of the generating curves.

Before describing the solutions to \eqref{system1}-\eqref{system3} with initial conditions \eqref{conditions}, we need to clarify their existence. Notice that the system \eqref{system1}-\eqref{system3} is singular at $r=0$ and, hence, the existence of solution for the initial value problem posed above is not guaranteed by the standard theory of ordinary differential equations. Nevertheless, using an argument from \cite{CDO}, in Proposition 2.2 of \cite{PP3} the existence was shown as well as the continuous dependence on the parameters $c_o$ and $z_0$. We state the result here for the sake of clarity.

\begin{proposition}\label{existence} 
	Let $c_o> 0$ and $z_0\neq 0$. Then, the system of first order differential equations \eqref{system1}-\eqref{system3} with initial conditions \eqref{conditions} has analytic solutions that depend continuously on the parameters $c_o$ and $z_0$.
\end{proposition}

We now proceed with the description and classification of the curves $\gamma(s)=(r(s),z(s))$ constructed from solutions of \eqref{system1}-\eqref{system3} with initial conditions \eqref{conditions}. We will first obtain an integral of equation \eqref{system3}. Multiplying \eqref{system3} by $r\cos\varphi$, we obtain
$$\frac{d}{ds}\left(r(s)\sin\varphi(s)+c_o r^2(s)\right)=-2\frac{r(s)\cos^2\varphi(s)}{z(s)}\,.$$
Hence, integrating this from $s=0$ to $s$, and taking into account that $r(0)=0$, we have
\begin{equation}\label{eqi}
	r(s)\left(\sin\varphi(s)+r(s)c_o\right)=-2\int_0^s\frac{r(t)\cos^2\varphi(t)}{z(t)}\,dt\,.
\end{equation}

The classification of the curves $\gamma(s)=(r(s),z(s))$ will depend on the sign of $\varphi'(0)$. Applying L'H\^{o}pital's rule to \eqref{system3} and using the initial conditions \eqref{conditions}, we have at $s=0$
\begin{equation}\label{ai}
	\varphi'(0)=-\frac{1}{z_0}-c_o\,.
\end{equation}
Therefore, there exist three essentially different cases, namely, $z_0>0$, $-1/c_o<z_0<0$, and $z_0<-1/c_o$. We refer to the curves of these cases as \emph{unduloid-type}, \emph{ovaloid-type}, and \emph{nodoid-type}, respectively.

\begin{rem}\label{line}
	From the existence and uniqueness of solution to the system \eqref{system1}-\eqref{system3} with initial conditions \eqref{conditions}, the case $z_0=-1/c_o$ corresponds with a horizontal straight line (see the fifth curve of Figure \ref{Profiles}).
\end{rem}

We first consider the unduloid-type curves (that is, the case $z_0>0$).

\begin{theorem}\label{tc1}
Let $c_o>0$, $z_0>0$ and consider the solution $\{r(s),z(s),\varphi(s)\}$ of \eqref{system1}-\eqref{system3} with initial conditions \eqref{conditions}. Then, the curve $\gamma(s)=(r(s),z(s))$ is a graph over the $z$-axis, $r= r(z)<1/c_o$, defined on $[0,z_0]$ and which intersects the $r$-axis orthogonally.
\end{theorem}
\begin{proof} Let $\{r(s),z(s),\varphi(s)\}$ be the solution of \eqref{system1}-\eqref{system3} with initial conditions \eqref{conditions} for an initial height $z_0>0$. It then follows from \eqref{ai} that $\varphi'(0)<0$ and, hence, $\varphi(s)$ decreases near $s=0$. Therefore, since $\varphi(0)=0$ and $\varphi$ decreases for small values of $s$, we have that $\varphi(s)<0$ and that $\sin\varphi(s)<0$ for every $s\in(0,\epsilon)$ where $\epsilon>0$ is sufficiently small. Denote by $[0,\ell)$ the maximal domain of the solution. Since $z(s)>0$ for every $s\in[0,\ell)$, it follows from \eqref{eqi} that $r(s)>0$ for every $s\in(0,\ell)$. On the contrary, the left-hand side of \eqref{eqi} would be zero while the right-hand side would be negative. In addition, the same reasoning proves that
	$$\lim_{s\to\ell} r(s)\neq 0\,.$$
Therefore, for the fixed $\epsilon>0$ as above, there exists a constant $\lambda>0$ (depending only on $\epsilon$) such that $r(s)\geq \lambda$ for every $s\in(\epsilon,\ell)$.

Similarly, we also deduce from \eqref{eqi} that $\sin\varphi(s)+c_or(s)<0$ holds for every $s\in(0,\ell)$. This implies that
$$r(s)<-\frac{\sin\varphi(s)}{c_o}\leq \frac{1}{c_o}\,.$$
Consequently, we have that $z'(s)=\sin\varphi(s)<-c_or(s)\leq -c_o\lambda<0$ for every $s\in(\epsilon,\ell)$, which implies that the function $z(s)$ decreases everywhere in its domain and $z(s)\to 0$ as $s\to\ell$. In particular, since $z(s)$ is strictly decreasing, the curve $\gamma(s)=\left(r(s),z(s)\right)$ is a graph over the $z$-axis.

Finally, we will prove that the curve $\gamma(s)=(r(s),z(s))$ intersects the $r$-axis orthogonally or, equivalently, $\varphi(s)\to-\pi/2$ as $s\to \ell$. Since $\gamma$ is a graph $r=r(z)$, we rewrite \eqref{system3} as
$$z^2\left(\frac{r_z}{z^2\sqrt{1+r_z^2}}\right)_z=\frac{1}{r\sqrt{1+r_z^2}}-2c_o\,.$$
Dividing by $z^2$ and integrating from $z$ to $z_0$, we have
$$\frac{1}{z_0^2}-\frac{r_z}{z^2\sqrt{1+r_z^2}}=\int_z^{z_0}\frac{dz}{rz^2\sqrt{1+r_z^2}}+2c_o\left(\frac{1}{z_0}-\frac{1}{z}\right),$$
since, due to the initial conditions at $s=0$,
$$\lim_{z\to z_0}\frac{r_z}{z^2\sqrt{1+r_z^2}}=\lim_{z\to z_0}\frac{1}{z^2\sqrt{1+z_r^2}}=\frac{1}{z_0^2}\,.$$
Multiplying by $z^2$ and rearranging, we obtain
$$\frac{r_z}{\sqrt{1+r_z^2}}=-z^2\int_z^{z_0}\frac{dz}{rz^2\sqrt{1+r_z^2}}-2c_o\left(\frac{z}{z_0}-1\right)z+\frac{z^2}{z_0^2}\,.$$
Let $z_\epsilon=z(\epsilon)\in(0,z_0)$. Recall that there exists a positive constant $\lambda>0$ such that $r(z)\geq \lambda$ for every $z\in(z_\epsilon,z_0)$. It then follows that
$$\left| \frac{r_z}{\sqrt{1+r_z^2}} \right|\leq \frac{1}{\lambda}\left(1-\frac{z}{z_0}\right)z+2c_o\left(1-\frac{z}{z_0}\right)z+\frac{z^2}{z_0^2}\longrightarrow 0\,,$$
when $z\to 0$. That is, $r_z(0)=0$, which proves that $\gamma$ meets the $r$-axis orthogonally.
\end{proof}

\begin{rem} When $z_0>0$ is small, the curve $\gamma$ is a graph over the $r$-axis as well. As $z_0>0$ increases, there appear points in $\gamma$ where the radius attains their local maxima and minima. The value of the radius at those points oscillates around $1/(2c_o)$, which corresponds with the radius of the critical cylinder (see Remark \ref{cylinder}). Three of these curves are shown on the first row of Figure \ref{Profiles}.
\end{rem}

We next consider the ovaloid-type curves (that is, the case $z_0\in (-1/c_o,0)$).

\begin{theorem}\label{TypeII}
	Let $c_o>0$, $z_0\in(-1/c_o,0)$ and consider the solution $\{r(s),z(s),\varphi(s)\}$ of \eqref{system1}-\eqref{system3} with initial conditions \eqref{conditions}. Then, the curve $\gamma(s)=(r(s),z(s))$ is a convex graph (over both the $z$-axis and $r$-axis) which intersects the $r$-axis orthogonally.
\end{theorem}
\begin{proof} Let $\{r(s),z(s),\varphi(s)\}$ be the solution of \eqref{system1}-\eqref{system3} with initial conditions \eqref{conditions} for an initial height $z_0\in(-1/c_o,0)$. Denote by $[0,\ell)$ the maximal domain of this solution. We will show that $\varphi$ is strictly increasing for every $s\in[0,\ell)$ and that
	$$\lim_{s\to \ell} \varphi(s)=\frac{\pi}{2}\,,\quad\quad\quad \lim_{s\to \ell} z(s)=0\,,$$
hold.

We first observe that $\pi/2$ is not in the range of $\varphi$. Suppose by contradiction that there exists a value $s_0\in[0,\ell)$ such that $\varphi(s_0)=\pi/2$. Take it to be the first such a value, then $\varphi'(s_0)\geq 0$ must hold since $\varphi'(0)>0$ holds from \eqref{ai} and so $\varphi$ increases near $s=0$, at which $\varphi(0)=0$. However, evaluating \eqref{system3} at $s=s_0$ we get,
$$\varphi'(s_0)=-\frac{1}{r(s_0)}-2c_o<0\,,$$
which leads to a contradiction.
	
We will then show the monotonicity of $\varphi$. Recall that $\varphi(s)$ increases near $s=0$. Assume that $s_1\in[0,\ell)$ is the first value at which $\varphi'$ vanishes. In particular, $\varphi(s_1)>0$ holds. Differentiating \eqref{system3} we obtain
\begin{equation}\label{d2}
\varphi''(s)=\sin\varphi(s)\cos\varphi(s)\left(\frac{2}{z^2(s)}+\frac{1}{r^2(s)}\right)+\varphi'(s)\left(\frac{2\sin\varphi(s)}{z(s)}-\frac{\cos\varphi(s)}{r(s)}\right)\,,
\end{equation}
and, hence, at $s=s_1$ we have
$$\varphi''(s_1)=\sin\varphi(s_1)\cos\varphi(s_1)\left(\frac{2}{z^2(s_1)}+\frac{1}{r^2(s_1)}\right)>0\,,$$
since by previous observation $\varphi(s_1)<\pi/2$. This would imply that at $s=s_1$ the function $\varphi$ attains a local minimum, but this is not possible since $\varphi$ increases when $s\in[0,s_1)$. Thus, there cannot be a value $s_1\in[0,\ell)$ such that $\varphi'(s_1)=0$. Therefore, the function $\varphi$ is strictly increasing and its range is contained in $[0,\pi/2)$. Since $\varphi(s)\in (0,\pi/2)$ for every $s\in (0,\ell)$, then $r'(s)=\cos\varphi(s)\neq 0$ and $z'(s)=\sin\varphi(s)\neq 0$, proving that $\gamma$ is a graph over both the $r$-axis and $z$-axis. Moreover, since $\varphi'(s)>0$ for every $s\in[0,\ell)$, this graph is convex. 
 
In particular, since the range of $\varphi$ is contained in $[0,\pi/2)$ and $\varphi$ increases everywhere, we obtain that
$$\lim_{s\to \ell}\varphi(s)=\lambda\in(0,\pi/2]\,,$$
and, hence, the function $z'(s)=\sin\varphi(s)\to \sin \lambda>0$ when $s\to \ell$. This shows that $z(s)\to 0$ as $s\to \ell$. Then, an argument as in the last part of the proof of Theorem \ref{tc1} shows that $\varphi(s)\to \pi/2$ when $s\to \ell$. 
\end{proof}
	
Finally, we briefly mention about the nodoid-type curves (that is, the case $z_0<-1/c_o$). These curves do not intersect the $r$-axis and, hence, they will not be used throughout the paper. For this reason, we omit the proof here. The geometric description of part of these curves was formally shown in Theorem 2.2 of \cite{PP3}. Extending this argument, the following remark can be shown.

\begin{rem}
	Let $c_o>0$, $z_0<-1/c_o$ and consider the solution $\{r(s),z(s),\varphi(s)\}$ of \eqref{system1}-\eqref{system3} with initial conditions \eqref{conditions}. Then, the curve $\gamma(s)=\left(r(s),z(s)\right)$, which is defined for every $s\in[0,\infty)$, is a convex curve with self-intersections contained in $\{z<-1/c_o\}$. The last curve of Figure \ref{Profiles} represents this type.
\end{rem}

In Figure \ref{Profiles} we illustrate the one-parameter family of generating curves $\gamma(s)=(r(s),z(s))$ constructed from solutions of \eqref{system1}-\eqref{system3} with initial conditions \eqref{conditions}. As the initial height $z_0\neq 0$ varies we see that the type of curves passes through the three different cases studied above.

\begin{figure}[h!]
	\centering
	\includegraphics[height=5.4cm]{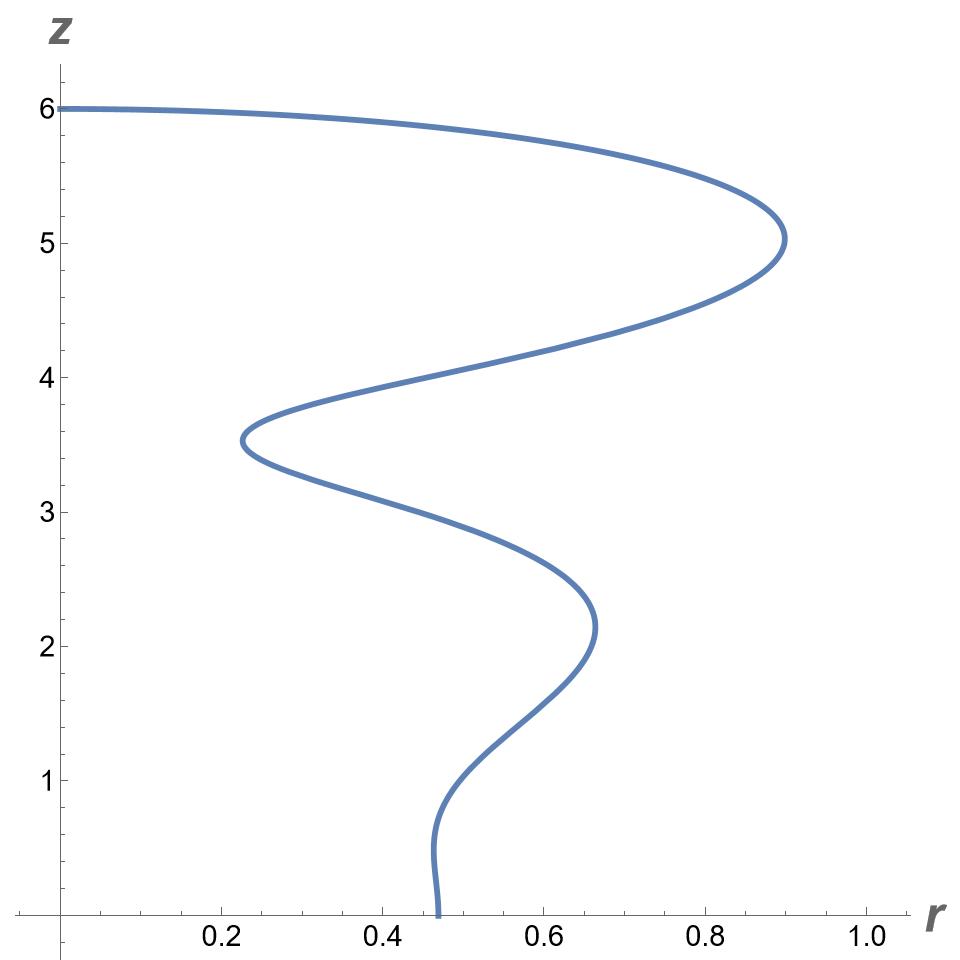}\,\,\,\,
	\includegraphics[height=5.4cm]{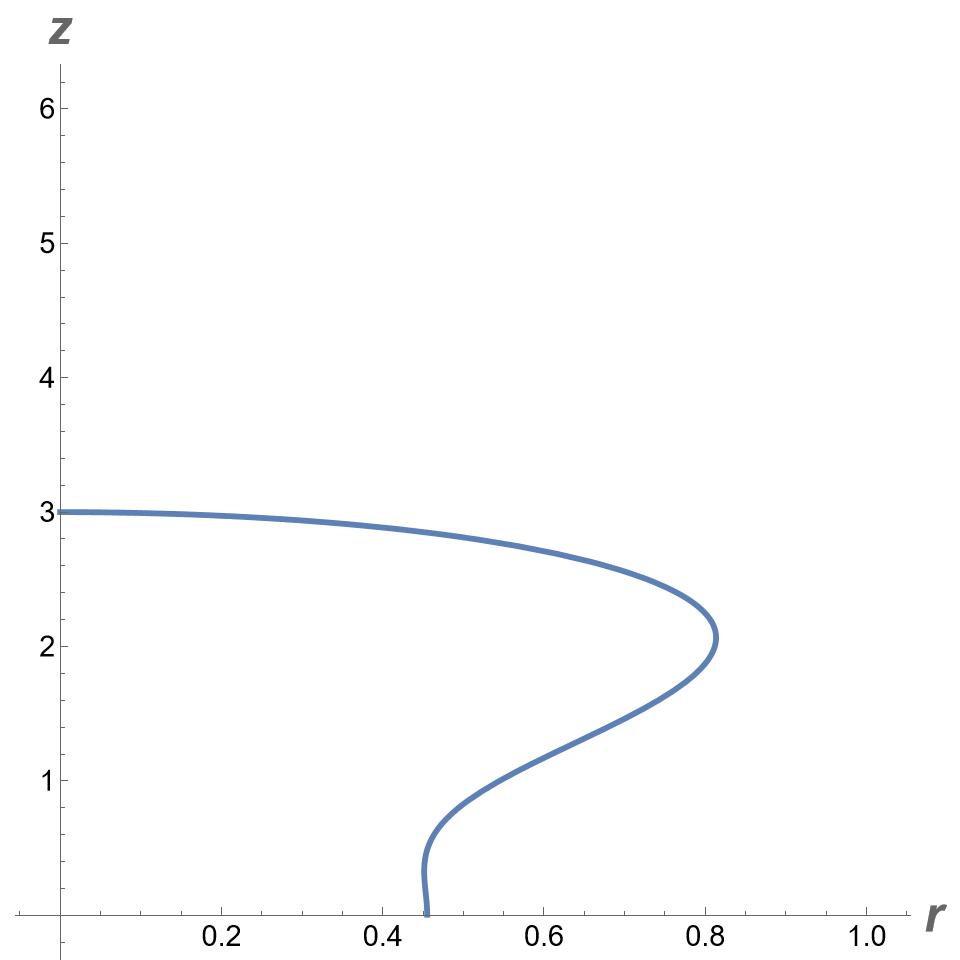}\,\,\,\,
	\includegraphics[height=5.4cm]{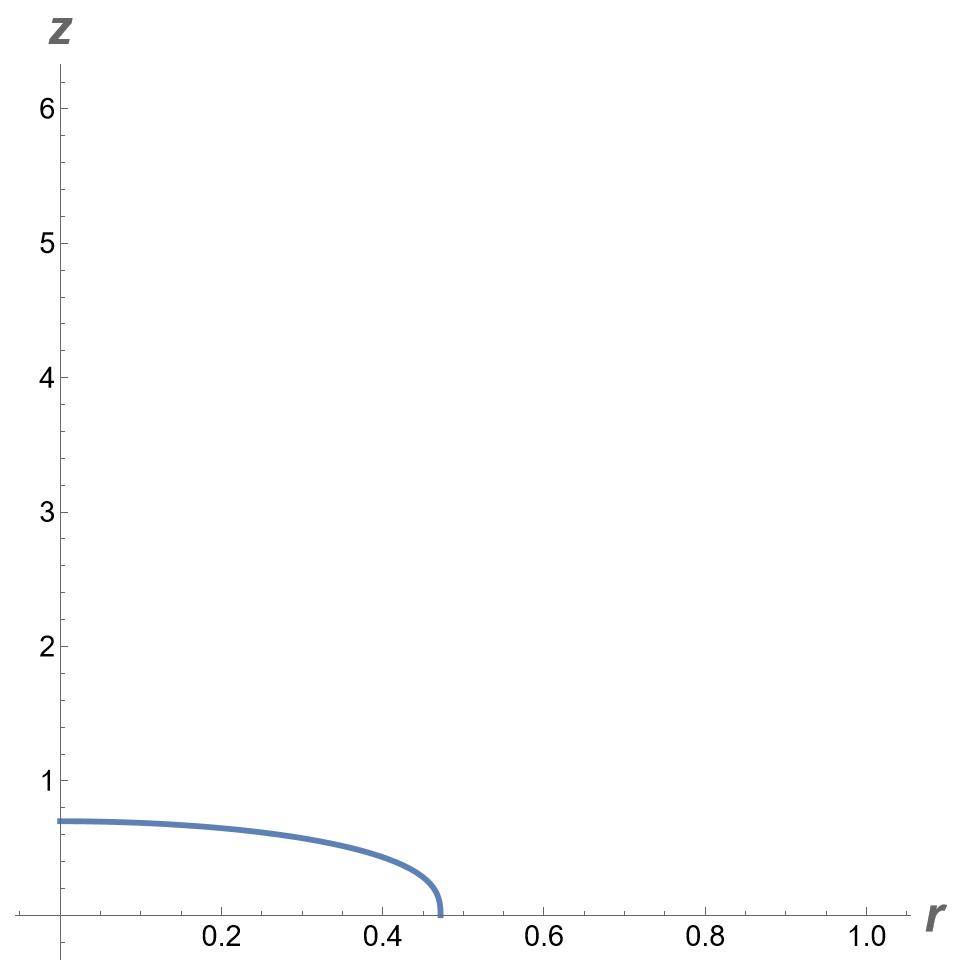}\vspace{0.75cm}
	\includegraphics[height=5.4cm]{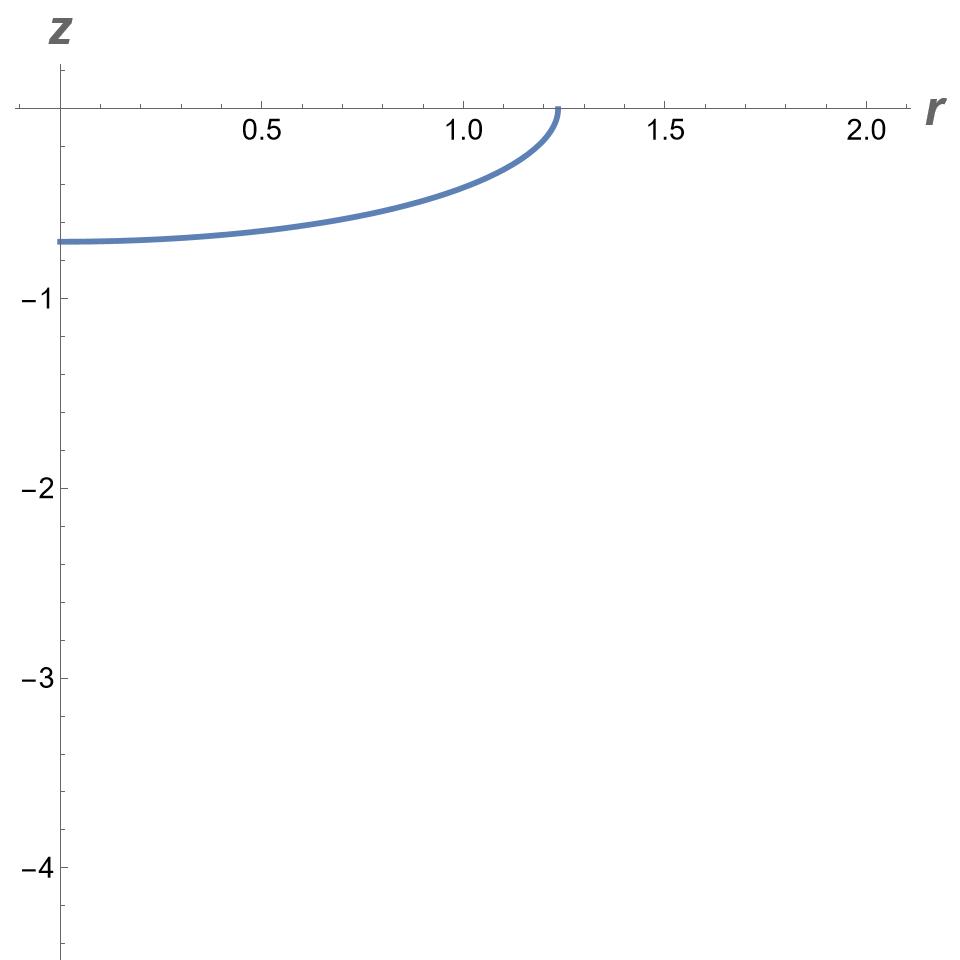}\,\,\,\,
	\includegraphics[height=5.4cm]{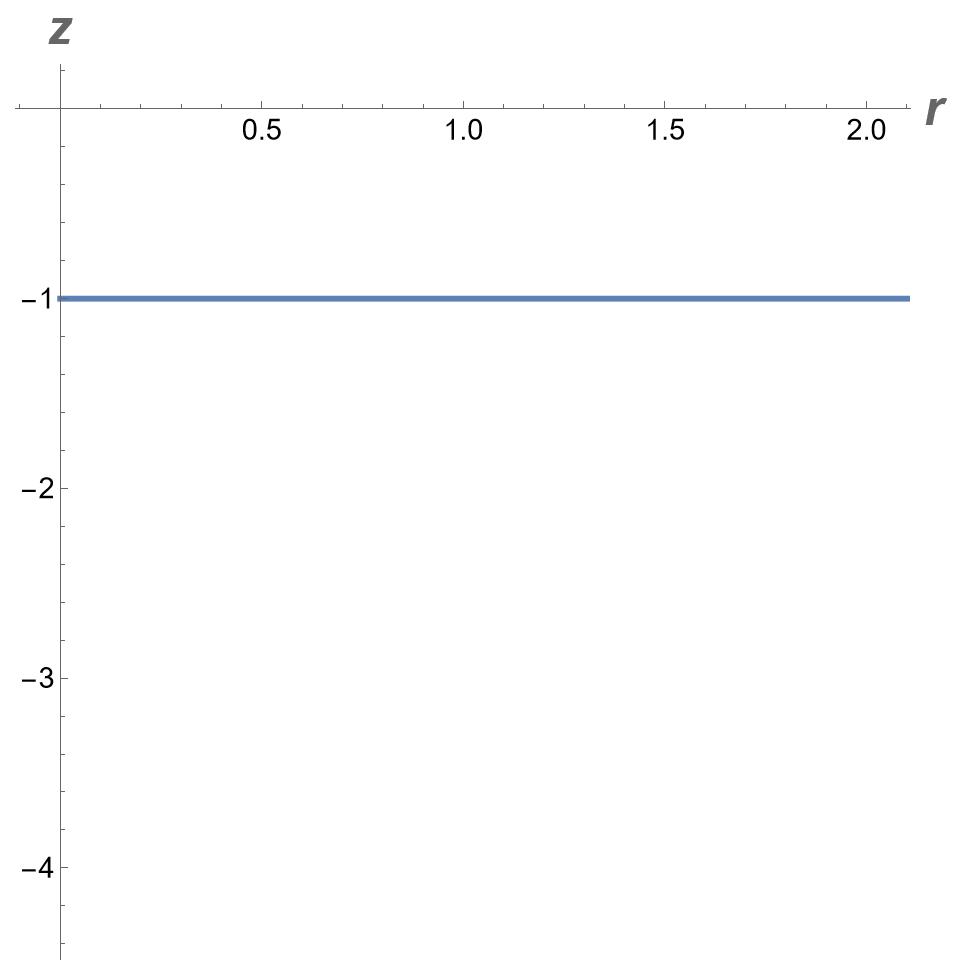}\,\,\,\,
	\includegraphics[height=5.4cm]{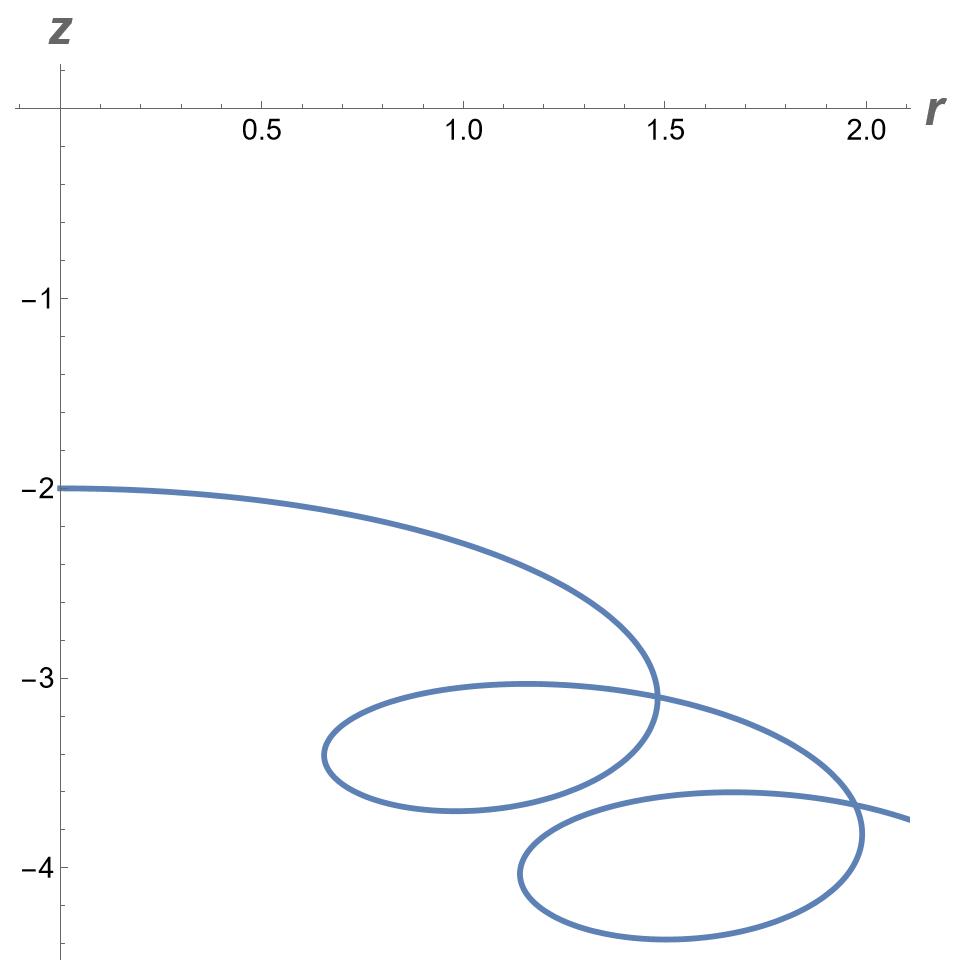}
	\caption{{\small Six curves $\gamma(s)=(r(s),z(s))$ constructed from solutions of \eqref{system1}-\eqref{system3} with initial conditions \eqref{conditions} and different initial heights $z_0\neq 0$. For all the curves we have fixed $c_o=1$. The curves on the top row are unduloid-type curves, the first curve on the bottom row is an ovaloid-type curve, and the last curve is a nodoid-type curve. The curve in the center of the bottom row is the horizontal straight line at height $z=-1/c_o$ (c.f., Remark \ref{line}). (Observe that the representations of the top and bottom rows have different scale.)}}
	\label{Profiles}
\end{figure}

\subsection{Construction and Regularity of Closed Surfaces}

In the previous paragraph, we obtained that, when either $c_o=0$ or $c_o>0$ and the initial heights $z_0\neq 0$ of the initial conditions \eqref{conditions} satisfy $z_0>-1/c_o$, the generating curves $\gamma(s)=(r(s),z(s))$ constructed from solutions of \eqref{system1}-\eqref{system3} intersect orthogonally the $r$-axis (see Remark \ref{circle} for the case $c_o=0$, Theorem \ref{tc1} for unduloid-type curves, and Theorem \ref{TypeII} for ovaloid-type curves). In particular, it follows from Theorem \ref{z=0} that meeting orthogonally the plane $\{z=0\}$ is a necessary condition to obtain closed surfaces satisfying \eqref{RME0}. Since $z=0$ is a singularity for the equation \eqref{system3}, one cannot extend the generating curves further. The procedure to construct axially symmetric surfaces satisfying \eqref{RME0} then consists of suitably gluing together along the plane $\{z=0\}$ two axially symmetric disc type surfaces, possibly one of which should be reflected. Recall that the reduced membrane equation \eqref{RME0} is invariant by the reflection across the plane $\{z=0\}$ (c.f., Proposition \ref{reflection}).

More precisely, let $c_o\geq 0$ be fixed. For a solution of \eqref{system1}-\eqref{system3} with initial conditions \eqref{conditions}, we can construct axially symmetric disc type surfaces simply by rotating (a part of) the curve $\gamma(s)=(r(s),z(s))$ around the $z$-axis. When either $c_o=0$ or $c_o>0$ and the initial height $z_0$ satisfies $z_0>-1/c_o$, the generating curve $\gamma$ meets the $r$-axis, say at $s=\ell$. We then rotate the part of the curve corresponding to $s\in[0,\ell]$. The point where the generating curve meets the $r$-axis will correspond to the boundary of the surface, which is a geodesic parallel contained in the plane $\{z=0\}$. It is then possible to glue this surface to its reflection across the plane $\{z=0\}$ (which satisfies the same reduced membrane equation \eqref{RME0}, as shown in Proposition \ref{reflection}). Of course, solutions of \eqref{system1}-\eqref{system3} with initial conditions \eqref{conditions} for different initial heights may as well be considered and the corresponding axially symmetric disc type surfaces glued together (if necessary, reflecting one of them) as long as these curves meet the $r$-axis at the same point. This gluing procedure then results in closed axially symmetric surfaces of genus zero that satisfy the reduced membrane equation \eqref{RME0} everywhere.

In the following result we study the regularity of the closed surfaces obtained from the above gluing procedure. It turns out that these surfaces are real analytic everywhere but, perhaps, on the geodesic parallel at height $z=0$.

\begin{theorem}\label{regularity}
	Let $X_+:\Sigma_+\longrightarrow\r^3$ and $X_-:\Sigma_-\longrightarrow\r^3$ be two axially symmetric immersions of disc type surfaces satisfying \eqref{RME0}. Assume, possibly after reflecting across the plane $\{z=0\}$ one of the surfaces, that:
	\begin{enumerate}
		\item $X_+\left(\Sigma_+\right)\subset\{z\geq 0\}$,
		\item $X_-\left(\Sigma_-\right)\subset\{z\leq 0\}$, and
		\item $X_+(\partial\Sigma_+)=X_-(\partial\Sigma_-)\subset\{z=0\}$.
	\end{enumerate}
	Then, $X:\Sigma_+\cup\Sigma_-\longrightarrow\r^3$, defined by $X\lvert_{\Sigma_+}=X_+$ and $X\lvert_{\Sigma_-}=X_-$, is the axially symmetric immersion of a closed surface of genus zero satisfying \eqref{RME0} everywhere. The immersion $X$ is real analytic wherever $z\neq 0$ and is of class, at least, $\mathcal{C}^2$ along the parallel $z=0$.
\end{theorem}
\begin{proof} 
Let $X:\Sigma_+\cup\Sigma_-\longrightarrow\r^3$ be the axially symmetric immersion obtained from $X_-$ and $X_+$. These surfaces satisfy the reduced membrane equation \eqref{RME0} and, hence, by the gluing procedure explained above $X(\Sigma_+\cup\Sigma_-)$ is a closed surface of genus zero satisfying \eqref{RME0} everywhere. This shows the first assertion. 

With respect to regularity, wherever $z\neq 0$ and $r\neq 0$, by writing equation \eqref{RME0} in a nonparametric form we get a second order elliptic partial differential equation with coefficients that depend analytically on the height function and its derivatives. It then follows by elliptic regularity \cite{M} that the immersion $X$ is real analytic wherever $z\neq 0$ and $r\neq 0$ holds.

For the cases $r=0$ and $z=0$, since the immersion $X$ is axially symmetric, it is enough to study the regularity for the generating curves $\gamma_+$ and $\gamma_-$, respectively. For simplicity with the signs (c.f., Remark \ref{+-}), we will assume after reflecting $\gamma_-$ across the $r$-axis that both $\gamma_+$ and $\gamma_-$ are contained in the half-plane $\{z\geq 0\}$ and that their arc length parameter $s\in[0,\ell_\pm]$ is measured beginning from the top.

We now consider what happens at $r=0$, which is a singularity of \eqref{RME0} (c.f., \eqref{system3}) and, hence, the classical theory does not apply. Throughout this part, to simplify the notation, we will avoid writing the subindexes in the functions because the argument works the same for $\gamma_+$ and $\gamma_-$. In this case, the generating curve $\gamma\equiv\gamma_\pm$ is a solution of \eqref{system1}-\eqref{system3} with initial conditions \eqref{conditions}. To check the regularity at $r=0$ we first extend the curve $\gamma$ to the values $r<0$, reflecting it across the $z$-axis. The reflection of $\gamma(s)=(r(s),z(s))$ across the $z$-axis is defined by $\widetilde{\gamma}(s)=(\widetilde{r}(s),\widetilde{z}(s))$, where
\begin{equation*}
	\widetilde{r}(s)=-r(-s)\,,\quad\quad\quad
	\widetilde{z}(s)=z(-s)\,,\quad\quad\quad
	\widetilde{\varphi}(s)=-\varphi(-s)\,.
\end{equation*}
From Proposition \ref{reflection}, reflections across vertical planes preserve \eqref{RME0} and, hence, it follows that $\{\widetilde{r}(s),\widetilde{z}(s),\widetilde{\varphi}(s)\}$ satisfies \eqref{system1}-\eqref{system3}. Due to the initial conditions \eqref{conditions}, the cut with the $z$-axis is orthogonal and, hence, $X$ is of class $\mathcal{C}^1$ at $r=0$. In fact, we have $\widetilde{\gamma}(0)=(0,z_0)=\gamma(0)$ and $\widetilde{\gamma}'(0)=(1,0)=\gamma'(0)$, since $\widetilde{\varphi}(0)=-\varphi(0)=0$. For the second derivative, we have
$$\gamma''(s)=\left(-\varphi'(s)\sin\varphi(s),\varphi'(s)\cos\varphi(s)\right),$$
and
$$\widetilde{\gamma}''(s)=\left(\varphi'(-s)\sin\varphi(-s),\varphi'(-s)\cos\varphi(-s)\right).$$
Using that $\varphi(0)=0$, we obtain that $\gamma''(0)=\widetilde{\gamma}''(0)$ as long as $\varphi'(0)$ is well defined. From \eqref{ai}, we have $\varphi'(0)=-1/z_0-c_o$. We then conclude that the immersion $X$ belongs to $\mathcal{C}^2$ at $r=0$. Differentiating $\gamma$ and $\widetilde{\gamma}$ once again, we get
$$\gamma'''(s)=\left(-\varphi''(s)\sin\varphi(s)-\varphi'(s)^2\cos\varphi(s),\varphi''(s)\cos\varphi(s)-\varphi'(s)^2\sin\varphi(s)\right),$$
while,
$$\widetilde{\gamma}'''(s)=\left(-\varphi''(-s)\sin\varphi(-s)-\varphi'(-s)^2\cos\varphi(-s),-\varphi''(-s)\cos\varphi(-s)+\varphi'(-s)^2\sin\varphi(-s)\right).$$
Since $\varphi(0)=0$ and $\varphi'(0)=-1/z_0-c_o$, we deduce that $\gamma'''(0)=\widetilde{\gamma}'''(0)$ if and only if $\varphi''(0)=0$. Differentiating \eqref{system3} and using the initial conditions \eqref{conditions}, we get
$$\varphi''(0)=\lim_{s\to 0}\left(-\frac{\varphi'(s)\cos\varphi(s)}{r(s)}+\frac{\sin\varphi(s)\cos\varphi(s)}{r^2(s)}\right)=\lim_{s\to 0} \frac{-\varphi'(s)r(s)+\sin\varphi(s)}{r^2(s)}=-\frac{\varphi''(0)}{2}\,,$$
where in the last equality we have applied L'H\^{o}pital's rule again. This shows that $\varphi''(0)=0$ and, therefore, $X$ belongs to $\mathcal{C}^3$ at $r=0$. We now recall that from Theorem \ref{z=0}, a $\mathcal{C}^3$ immersion satisfying \eqref{RME0} is a Helfrich surface and, hence, real analytic (c.f., Remark \ref{elliptic}). Our immersion $X$ belongs to $\mathcal{C}^3$ around $r=0$ and so it is real analytic there.

We study now the regularity of the immersion $X$ at $z=0$. In order to do that, we will proceed in a similar way to the case $r=0$. Observe that for simplicity with the signs we are assuming that both $\gamma_+$ and $\gamma_-$ are contained in the half-plane $\{z\geq 0\}$. Consequently, we need to study the regularity with what $\gamma_+$ meets $\widetilde{\gamma}_-$, i.e., the reflection of $\gamma_-$ across the $r$-axis. The curve $\widetilde{\gamma}_-$ is defined by $\widetilde{\gamma}_-(s)=(\widetilde{r}_-(s),\widetilde{z}_-(s))$ where
\begin{equation*}
	\widetilde{r}_-(s)=r_-(\ell_--s)\,,\quad\quad\quad
	\widetilde{z}_-(s)=-z_-(\ell_--s)\,,\quad\quad\quad
	\widetilde{\varphi}_-(s)=-\pi-\varphi_-(\ell_--s)\,.
\end{equation*}
Observe that the arc length parameter of $\widetilde{\gamma}_-$ is measured beginning on the $r$-axis, while that of $\gamma_+$ ends on the $r$-axis. Since both $\gamma_+$ and $\gamma_-$ meet the $r$-axis orthogonally at the same point we have $\gamma_+(\ell_+)=(r_*,0)=\widetilde{\gamma}_-(0)$ and $\gamma'_+(\ell_+)=(0,-1)=\widetilde{\gamma}'_-(0)$. This shows that $X$ is of class $\mathcal{C}^1$ at $z=0$. The second derivatives of $\gamma_+$ and $\widetilde{\gamma}_-$ are
$$\gamma''_+(s)=\left(-\varphi_+'(s)\sin\varphi_+(s),\varphi'_+(s)\cos\varphi_+(s)\right),$$
and
$$\widetilde{\gamma}_-''(s)=\left(-\varphi_-'(\ell_--s)\sin\varphi_-(\ell_--s),-\varphi_-'(\ell_--s)\cos\varphi_-(\ell_--s)\right).$$
Using that $\varphi_+(\ell_+)=\varphi_-(\ell_-)=-\pi/2$, we conclude that $\gamma_+''(\ell_+)=\widetilde{\gamma}_-''(0)$ if and only if $\varphi_+'(\ell_+)=\varphi_-'(\ell_-)$ and they are well defined. Applying L'H\^{o}pital's rule to \eqref{system3}, we get $\varphi_+'(\ell_+)=2c_o-1/r_*=\varphi'_-(\ell_-)$. Consequently, the immersion $X$ is of class $\mathcal{C}^2$ at $z=0$. 
\end{proof}

Computing the third derivative of $\gamma_+$ and $\widetilde{\gamma}_-$ at $z=0$ (see the proof of Theorem \ref{regularity}) we deduce the following consequence.  
    
\begin{corollary}\label{H'}
Let $X:\Sigma_+\cup\Sigma_-\longrightarrow\r^3$ be the axially symmetric immersion constructed with the gluing procedure of Theorem \ref{regularity}. Then $X$ is of class $\mathcal{C}^3$ at $z=0$  (and, hence, real analytic), if and only if $\partial_{n_+}H=-\partial_{n_-}H$ at $z=0$, where $n_+$ and $n_-$ are the outward pointing conormals to $\partial\Sigma_+$ and $\partial\Sigma_-$, respectively.
\end{corollary} 
\begin{proof}
Consider the axially symmetric immersion $X:\Sigma_+\cup\Sigma_-\longrightarrow\mathbb{R}^3$ obtained in Theorem \ref{regularity} from $X_+:\Sigma_+\longrightarrow\mathbb{R}^3$ and $X_-:\Sigma_-\longrightarrow\mathbb{R}^3$ and denote by $\gamma_+$ and $\gamma_-$ the generating curves of the top $X_+(\Sigma_+)$ and bottom $X_-(\Sigma_-)$ disc type surfaces, respectively. In addition, as already done in the proof of Theorem \ref{regularity}, assume for simplicity with the signs that both $\gamma_+$ and $\gamma_-$ are contained in $\{z\geq 0\}$ (for this we need to reflect first $\gamma_-$ across the $r$-axis).

Under this assumptions we are in the same setting as in Theorem \ref{regularity}. Hence, we can differentiate again $\gamma''_+(s)$ and $\widetilde{\gamma}''_-(s)$ (recall that $\widetilde{\gamma}_-$ is the reflection of $\gamma_-$ across the $r$-axis as defined in the proof of Theorem \ref{regularity}) obtaining
$$\gamma_+'''(s)=\left(-\varphi_+''(s)\sin\varphi_+(s)-\varphi'_+(s)^2\cos\varphi_+(s),\varphi_+''(s)\cos\varphi_+(s)-\varphi_+'(s)^2\sin\varphi_+(s)\right),$$
and
\begin{eqnarray*}
\widetilde{\gamma}_-'''(s)&=&(\varphi_-''(\ell_--s)\sin\varphi_-(\ell_--s)+\varphi_-'(\ell_--s)^2\cos\varphi_-(\ell_--s),\\
&&\varphi_-''(\ell_--s)\cos\varphi_-(\ell_--s)-\varphi_-'(\ell_--s)^2\sin\varphi_-(\ell_--s)).
\end{eqnarray*}
Using that $\varphi_+(\ell_+)=\varphi_-(\ell_-)=-\pi/2$ and $\varphi_+'(\ell_+)=\varphi'_-(\ell_-)=2c_o-1/r_*$ we deduce that $\gamma_+'''(\ell_+)=\widetilde{\gamma}_-'''(0)$ if and only if $\varphi_+''(\ell_+)=-\varphi_-''(\ell_-)$.

Finally, for both $X_\pm(\Sigma_\pm)$ the expression of the mean curvature is given by
\begin{equation}\label{H}
	H(s)=\frac{z'(s)}{2r(s)}-\frac{r''(s)}{2z'(s)}=\frac{\sin\varphi(s)}{2r(s)}+\frac{\varphi'(s)}{2}\,,
\end{equation}
for our choice of orientation and arc length parameter $s$ (measured beginning from the top of $\gamma\equiv\gamma_\pm$). Differentiating \eqref{H} and employing that $\varphi_+''(\ell_+)=-\varphi_-''(\ell_-)$, we obtain
$$H'(\ell_+)=\frac{\varphi_+''(\ell_+)}{2}=\frac{-\varphi_-''(\ell_-)}{2}=-H'(\ell_-)\,,$$
where, due to the choice of arc length parameter $s$, the derivative with respect to $s$ points in the outward direction. This finishes the proof.
\end{proof}

\section{Condition for Axially Symmetric Helfrich Spheres}

When the surface has sufficient regularity, the reduced membrane equation \eqref{RME0} is a sufficient condition for \eqref{EL} to hold. In particular, this is true if the immersion is of class $\mathcal{C}^3$ (c.f., Theorem \ref{z=0}). Unfortunately, even though the reduced membrane equation \eqref{RME0} is satisfied everywhere along the closed surfaces constructed in the previous section, those surfaces which are only $\mathcal{C}^2$ (and not $\mathcal{C}^3$) at the intersection with the plane $\{z=0\}$ are not Helfrich surfaces. They are not critical points of the Helfrich energy $\mathcal{H}$, not even in a weak sense (see Lemma \ref{2}  below).

In the following result we prove a necessary and sufficient condition for an axially symmetric closed surface of genus zero satisfying \eqref{RME0} everywhere to be a Helfrich surface.

\begin{theorem}\label{rescaling}
	Let $\Sigma$ be a closed genus zero surface and $X:\Sigma\longrightarrow\r^3$ an axially symmetric immersion with non-constant mean curvature. The immersion is critical for the Helfrich energy $\mathcal{H}$ if and only if, after a suitable rigid motion and translation of the vertical coordinate, it satisfies the reduced membrane equation \eqref{RME0} and the rescaling condition
\begin{equation}\label{rc}
\int_{\Sigma}\left(H+c_o\right)d\Sigma=0\,.
\end{equation}
\end{theorem}

\begin{rem}\label{rnew} Requesting that the mean curvature of the immersion is non-constant necessarily implies that $c_o\neq 0$. Indeed, if $c_o=0$, the only axially symmetric topological spheres critical for the Helfrich energy, which in this case is just the Willmore energy (recall that the Willmore energy is rescaling invariant), are round spheres. This is expected since the unique solution to \eqref{system1}-\eqref{system3} with initial conditions \eqref{conditions} is (a part of) a circle (c.f., Remark \ref{circle}) which after rotation around the $z$-axis and the gluing procedure of previous section gives rise to a round sphere of arbitrary radius (c.f., Remark \ref{cylinder}). On the other hand, when $c_o\neq 0$, the rescaling condition \eqref{rc} makes sense (see Theorem \ref{z=0}).
\end{rem}
  
The proof of Theorem \ref{rescaling} will follow from three lemmas. Throughout these lemmas, let $X:\Sigma\longrightarrow\r^3$ be an axially symmetric immersion with non-constant mean curvature of a closed surface of genus zero. Without loss of generality we may assume, after a rigid motion if necessary, that the axis of rotation is the $z$-axis.

We first show that if $X$ is critical for the Helfrich energy $\mathcal{H}$, then the reduced membrane equation \eqref{RME0} holds.

\begin{lemma}\label{1}
	Assume that the immersion $X$ is critical for the Helfrich energy $\mathcal{H}$. Then, after a suitable translation of the vertical coordinate, \eqref{RME0} holds on $\Sigma$.
\end{lemma}
\begin{proof} Cut the axially symmetric surface $X(\Sigma)$ with a suitable horizontal plane so that the part above this plane and close enough to the $z$-axis is an axially symmetric disc type surface, which we denote by $\Sigma_\epsilon$. Since the Euler-Lagrange equation \eqref{EL} was satisfied on the original surface $\Sigma$, it also holds on $\Sigma_\epsilon$.

Then, the flux argument of Theorem 4.1 of \cite{PP2} employed on $\Sigma_\epsilon$ shows that the reduced membrane equation \eqref{RME0} must hold on $\Sigma_\epsilon$, possibly after a suitable translation of the vertical coordinate. We point out here that to apply this argument we are implicitly using that $H\neq -c_o$ holds.

Recall that, by elliptic regularity, Helfrich surfaces are real analytic. Since on a part of $\Sigma$, namely on $\Sigma_\epsilon$, \eqref{RME0} holds, it follows from the real analyticity that the same equation must hold on the entire $\Sigma$. 
\end{proof}

Before stating the next lemma, we recall that a sufficiently regular immersion $X$ of a closed surface satisfying \eqref{RME0} must intersect the plane $\{z=0\}$ (see Theorem \ref{z=0}). We next prove that $\mathcal{C}^3$ regularity at the points where $z=0$ holds is equivalent to the criticality of the immersion. 

\begin{lemma}\label{2}
	Assume that the immersion $X$ satisfies \eqref{RME0}. Then, $X$ is critical for the Helfrich energy $\mathcal{H}$ if and only if $X$ is of class $\mathcal{C}^3$ at the parallel $z=0$.
\end{lemma}
\begin{proof} Let $\epsilon>0$ be sufficiently small and denote by $\Sigma_{\epsilon_+}$ the part of $\Sigma$ such that $X(\Sigma_{\epsilon_+})\subset\{z>\epsilon\}$. Similarly, let $\Sigma_{\epsilon_-}\subset\Sigma$ be such that $X(\Sigma_{\epsilon_-})\subset\{z<-\epsilon\}$. Then, for normal variations $\delta X=\psi\nu$, $\psi\in\mathcal{C}^\infty_o(\Sigma)$, we have (see Appendix A of \cite{PP1} for details)
\begin{eqnarray*}
	\delta\left(\int_{\Sigma_{\epsilon_+}\cup\Sigma_{\epsilon_-}}\left(H+c_o\right)^2d\Sigma\right)&=&\int_{\Sigma_{\epsilon_+}\cup\Sigma_{\epsilon_-}}\left(\Delta H+2(H+c_o)\left(H(H-c_o)-K\right)\right)\psi\,d\Sigma\\
	&&+\oint_{\partial\Sigma_{\epsilon_+}}\left((H+c_o)\partial_{n_+}\psi-\partial_{n_+}H\psi\right)d\sigma\\
	&&+\oint_{\partial\Sigma_{\epsilon_-}}\left((H+c_o)\partial_{n_-}\psi-\partial_{n_-}H\psi\right)d\sigma\,,
\end{eqnarray*}
where $n_\pm$ are the (outward pointing) conormals to $\partial\Sigma_{\epsilon_\pm}$, respectively, and $\partial_{n_\pm}$ represent the derivatives in these directions. Since on $\Sigma_{\epsilon_+}\cup\Sigma_{\epsilon_-}$ the reduced membrane equation \eqref{RME0} holds and, from Theorem \ref{regularity}, the surface is real analytic there, it follows that \eqref{EL} is also satisfied (Proposition 4.1 of \cite{PP2}). We then have only the two boundary integrals. The boundaries $\partial\Sigma_{\epsilon_\pm}$ are two parallel circles at heights $z=\pm\epsilon$ and of radius $r_{\epsilon_\pm}$, respectively. Due to the axial symmetry of $X$ the corresponding integrands are constant along these parallels. Hence, we obtain
\begin{eqnarray*}
	\delta\left(\int_{\Sigma_{\epsilon_+}\cup\Sigma_{\epsilon_-}}\left(H+c_o\right)^2d\Sigma\right)&=&2\pi r_{\epsilon_+}\left((H+c_o)\partial_{n_+}\psi-\partial_{n_+}H\psi\right)_{z=\epsilon_+}\\&&+2\pi r_{\epsilon_-}\left((H+c_o)\partial_{n_-}\psi-\partial_{n_-}H\psi\right)_{z=\epsilon_-}\,.
\end{eqnarray*}
Observe that as $\epsilon\to 0$, $r_{\epsilon_\pm}\to r_*$, the radius of the geodesic parallel at height $z=0$, and $H+c_o\to 2c_o-1/r_*$. The last assertion follows from \eqref{H} in combination with $\varphi(\ell)=-\pi/2$ and $\varphi'(\ell)=2c_o-1/r_*$ (c.f., last part of the proof of Theorem \ref{regularity}). Moreover, since at $\epsilon=0$, $n_+=-n_-$, the first and third terms above cancel out. Thus,
$$\delta\mathcal{H}[\Sigma]=\lim_{\epsilon\to 0} \delta\left(\int_{\Sigma_{\epsilon_+}\cup\Sigma_{\epsilon_-}}\left(H+c_o\right)^2d\Sigma\right)=-2\pi r_*\left(\partial_{n_+}H+\partial_{n_-}H\right)_{z=0}\psi\lvert_{z=0}\,.$$
Consequently, $X$ is a critical point for the Helfrich energy $\mathcal{H}$ if and only if, at $z=0$, $\partial_{n_+}H=-\partial_{n_-}H$. Recall that as mentioned in Corollary \ref{H'}, this condition is equivalent to $X$ belonging to $\mathcal{C}^3$ at $z=0$. 
\end{proof}

In the last lemma we show that, if $c_o\neq 0$, $\mathcal{C}^3$ regularity at $z=0$ is equivalent to the mean value of $H+c_o$ vanishing on the surface.

\begin{lemma}\label{3}
	Assume that the immersion $X$ satisfies \eqref{RME0} for $c_o\neq 0$. Then, $X$ is critical for the Helfrich energy $\mathcal{H}$ if and only if \eqref{rc} holds.
\end{lemma}
\begin{proof} Consider the domain $\Sigma_+\subset\Sigma$ to be the part of the surface such that $X(\Sigma_+)\subset\{z\geq 0\}$. Since \eqref{RME0} holds on $\Sigma_+$ and the singularity $z=0$ is not contained in the interior of $\Sigma_+$, \eqref{EL} holds. Thus, $X(\Sigma_+)$ is a Helfrich surface (with boundary). In \cite{PP2}, it was shown that for any Helfrich surface
$$2c_o\left(H+c_o\right)=\nabla\cdot\left[(H+c_o)\nabla q-q\nabla (H+c_o)+(H+c_o)^2\nabla X^2/2\right]=\nabla\cdot \mathcal{J}_X\,,$$
holds, where $\mathcal{J}_X$ denotes the vector field in the square brackets.  Here, $q:=X\cdot\nu$ is the support function. Integrating this over $\Sigma_+$ and applying the divergence theorem we get
$$2c_o\int_{\Sigma_+}\left(H+c_o\right)d\Sigma=\int_{\Sigma_+}\nabla\cdot\mathcal{J}_X\,d\Sigma=\oint_{\partial\Sigma_+}\mathcal{J}_X\cdot n_+\,d\sigma\,,$$
where $n_+$ is the (outward pointing) conormal to the boundary $\partial\Sigma_+$.

The boundary $\partial\Sigma_+$ is a geodesic parallel and, hence, along it $\nabla q=-\kappa_1 \nabla X^2/2$ and $\nabla X^2/2\cdot n_+=X\cdot n_+=0$. Consequently,
\begin{equation}\label{cond+}
	2c_o\int_{\Sigma_+}\left(H+c_o\right)d\Sigma=\oint_{\partial\Sigma_+}-q\,\partial_{n_+}\left(H+c_o\right)d\sigma=-2\pi r_*^2\,\partial_{n_+} \left(H+c_o\right)_{z=0},
\end{equation}
since along $\partial\Sigma_+$ we have that $q=X\cdot \nu=r_*$ is the radius of the boundary circle.

Repeating the same computations for the domain $\Sigma_-\subset\Sigma$ (defined as the part of the surface such that $X(\Sigma_-)\subset\{z\leq 0\}$) we obtain the analogous relation, namely,
\begin{equation}\label{cond-}
	2c_o\int_{\Sigma_-}\left(H+c_o\right)d\Sigma=-2\pi r_*^2 \,\partial_{n_-}\left(H+c_o\right)_{z=0},
\end{equation}
where $n_-$ is the (outward pointing) conormal to the boundary $\partial\Sigma_-$. Observe that $n_-=-n_+$.

Thus, from \eqref{cond+} and \eqref{cond-} we have
$$2c_o\int_\Sigma(H+c_o)\,d\Sigma=2c_o\int_{\Sigma_+}(H+c_o)\,d\Sigma+2c_o\int_{\Sigma_-}(H+c_o)\,d\Sigma=-2\pi r_*^2\left( \partial_{n_+}H+\partial_{n_-}H\right)_{z=0}.$$
From Lemma \ref{2}, $X$ is critical for the Helfrich energy $\mathcal{H}$ if and only if $X$ is of class $\mathcal{C}^3$ at $z=0$, which in turn is equivalent to $\partial_{n_+}H=-\partial_{n_-}H$ at $z=0$ (Corollary \ref{H'}). Thus, since $c_o\neq 0$, from the above relation we have an equivalence between criticality for the Helfrich energy and the condition regarding the mean value of $H+c_o$. \end{proof}

\begin{rem}\label{sphere} If $c_o=0$, an axially symmetric immersion $X:\Sigma\longrightarrow\mathbb{R}^3$ satisfying \eqref{RME0} is a round sphere (c.f., Remark \ref{rnew}).
\end{rem}

We can now proceed to show the statement of Theorem \ref{rescaling}.
\\

\noindent{\emph{Proof of Theorem \ref{rescaling}.} We begin by observing that since the mean curvature of the immersion $X$ is assumed to be non-constant, then $c_o\neq 0$ must hold. To the contrary, if $c_o=0$, by the explanation of Remark \ref{rnew}, the surface $X(\Sigma)$ would be a round sphere which has constant mean curvature, contradicting our request.
	
We next prove the forward implication. Assume that $X$ is critical for the Helfrich energy $\mathcal{H}$. By Lemma \ref{1}, the reduced membrane equation \eqref{RME0} holds on $\Sigma$. Moreover, from a rescaling argument of the Helfrich energy $\mathcal{H}$ as done in the proof of Theorem \ref{z=0} and because $c_o\neq 0$, we have that the rescaling condition \eqref{rc} is satisfied (see also Lemma \ref{3}).

For the converse, assume that \eqref{RME0} and \eqref{rc} are satisfied. From Lemma \ref{3} (which can be applied since $c_o\neq 0$ holds), we conclude that $X$ is critical for the Helfrich energy $\mathcal{H}$. This concludes the proof of the theorem. \hfill$\square$
\\

The result of Theorem \ref{rescaling} introduces a necessary and sufficient condition to determine whether or not closed axially symmetric surfaces satisfying the reduced membrane equation \eqref{RME0} for $c_o\neq 0$ are Helfrich surfaces. This condition is analytical in the sense that involves the computation of an integral, which represents the mean value of $H+c_o$ on the surface. Combining Lemmas \ref{2} and \ref{3}, this means that those surfaces which lack sufficient regularity at the parallel $z=0$ (they are $\mathcal{C}^2$ but not $\mathcal{C}^3$), are not critical points of $\mathcal{H}$.

\section{Examples of Axially Symmetric Helfrich Spheres}

In this section we will obtain examples of axially symmetric Helfrich spheres (other than round ones). For this, we are going to employ Theorem \ref{rescaling} and look for those closed solutions of the reduced membrane equation \eqref{RME0} for $c_o\neq 0$ that satisfy the rescaling condition \eqref{rc}. One possibility for \eqref{rc} to hold is that the mean value of $H+c_o$ on the top $\Sigma_+$ and bottom $\Sigma_-$ parts of the axially symmetric surface are both identically zero. 

Let $\Sigma$ be a closed genus zero surface and $X:\Sigma\longrightarrow\mathbb{R}^3$ an axially symmetric immersion satisfying the reduced membrane equation \eqref{RME0} for $c_o\neq 0$. Denote by $\Sigma_+\subset\Sigma$ the part of the surface such that $X(\Sigma_+)\subset\{z\geq 0\}$ and by $\Sigma_-\subset\Sigma$ the part such that $X(\Sigma_-)\subset\{z\leq 0\}$. If
\begin{equation}\label{hh}
\int_{\Sigma_+}(H+c_o)\,d\Sigma=\int_{\Sigma_-}(H+c_o)\,d\Sigma=0\,,
\end{equation}
holds, then \eqref{rc} is automatically satisfied and, hence, the immersion is critical for the Helfrich energy $\mathcal{H}$.

\begin{rem}\label{sec}
	From formulas \eqref{cond+} and \eqref{cond-}, it follows that the two conditions \eqref{hh} are equivalent to $\partial_{n_+} H=0=\partial_{n_-}H$ at $z=0$. In terms of the generating curve $\gamma\equiv\gamma_\pm$ this means that $\varphi''(\ell)=0$ (c.f., Corollary \ref{H'}).
\end{rem}

We next show that conditions \eqref{hh} imply that the surface is symmetric with respect to the plane $\{z=0\}$.

\begin{proposition}\label{symmetry}
	 Let $\Sigma$ be a closed genus zero surface and $X:\Sigma\longrightarrow\r^3$ an axially symmetric immersion satisfying \eqref{RME0} for $c_o\neq 0$. If \eqref{hh} hold, then the immersion is critical for the Helfrich energy $\mathcal{H}$ and $\{z=0\}$ is a plane of symmetry. Moreover, the vertical component $\nu_3$ of the unit normal $\nu$ has, at least, one change of sign in $X(\Sigma)\cap\{z>0\}$.
\end{proposition}
\begin{proof} 
Let $X:\Sigma\longrightarrow\mathbb{R}^3$ be an axially symmetric immersion of a genus zero surface satisfying \eqref{RME0} for $c_o\neq 0$ and assume that conditions \eqref{hh} also hold. It then follows that the rescaling condition \eqref{rc} holds as well and, from Theorem \ref{rescaling}, the surface is Helfrich. 

Recall that Theorem \ref{z=0} guarantees that $X(\Sigma)$ intersects the plane $\{z=0\}$. Assume by contradiction that $X(\Sigma)$ is not symmetric with respect to the plane $\{z=0\}$.  
Then,   $X(\Sigma_+)$ and its reflection $\widetilde{X}(\Sigma_+)$   across the plane $\{z=0\}$  generates a closed genus zero surface $X(\Sigma_+)\cup \widetilde{X}(\Sigma_+)$, symmetric with respect to $\{z=0\}$, and such that \eqref{RME0} holds everywhere. In addition, since by assumption the mean value of $H+c_o$ on $\Sigma_+$ is zero, the rescaling condition \eqref{rc} is satisfied for $X(\Sigma_+)\cup \widetilde{X}(\Sigma_+)$. Hence, applying Theorem \ref{rescaling} this surface is critical for the Helfrich energy and so it is real analytic. We then have two different real analytic surfaces, namely, $X(\Sigma_+)\cup \widetilde{X}(\Sigma_+)$ and $X(\Sigma)$, which coincide on the part contained in $\{z>0\}$. This is impossible and, hence, we have a contradiction.

For the final assertion of the statement, since \eqref{RME0} holds and the surface $X(\Sigma)$ satisfies
$$\int_{\Sigma_+}\left(H+c_o\right)d\Sigma=0\,,$$
we deduce that
$$\int_{\Sigma_+}\frac{\nu_3}{z}\,d\Sigma=0\,.$$
Thus, since $z\geq 0$ holds on $\Sigma_+$, it then follows that $\nu_3$ must have a change of sign in $X(\Sigma)\cap\{z>0\}$. 
\end{proof}

When $-1/c_o<z_0<0$, the solutions of \eqref{system1}-\eqref{system3} with initial conditions \eqref{conditions} are convex graphs over the $r$-axis (see Theorem \ref{TypeII}). Hence, on the (disc type) surface obtained rotating the corresponding curve $\gamma(s)=(r(s),z(s))$, $s\in[0,\ell]$ (recall that $s=\ell$ denotes the value of the arc length parameter at which $\gamma$ meets the $r$-axis), $\nu_3$ does not have a change of sign. Therefore, to get axially symmetric closed Helfrich surfaces of genus zero symmetric with respect to the plane $\{z=0\}$, it is enough to consider only positive values of the initial height $z_0$. In addition, the rescaling conditions \eqref{hh} must be satisfied. These conditions are equivalent to $\varphi''(\ell)=0$ (see Remark \ref{sec}). In conclusion, one needs to find the solutions of \eqref{system1}-\eqref{system3} for $c_o>0$ fixed with initial conditions \eqref{conditions} and a suitable $z_0>0$ such that $\varphi''(\ell)=0$. 

\begin{figure}[h!]
	\centering
	\includegraphics[height=12cm]{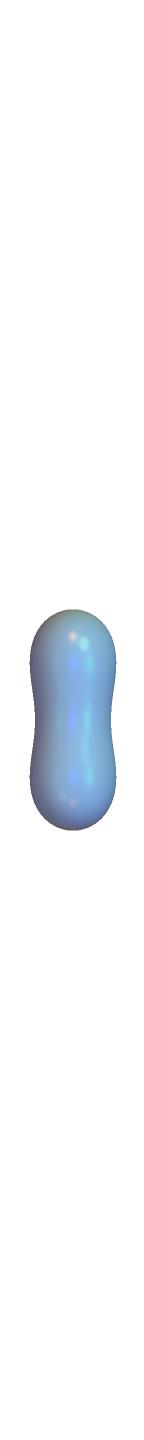}\quad\quad
	\includegraphics[height=12cm]{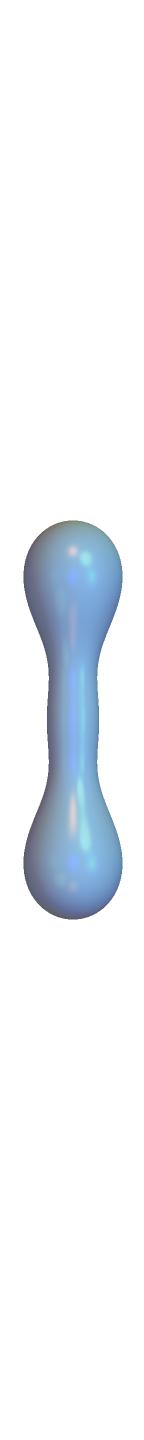}\quad\quad
	\includegraphics[height=12cm]{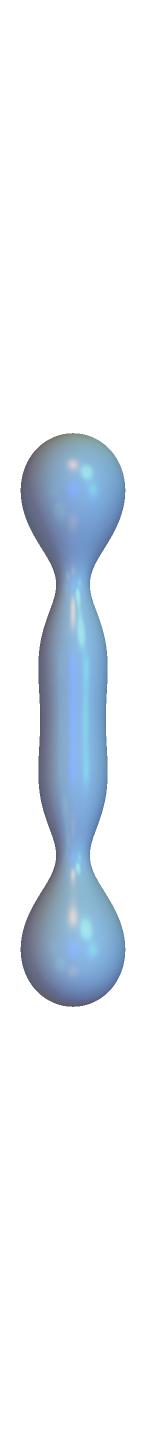}\quad\quad
	\includegraphics[height=12cm]{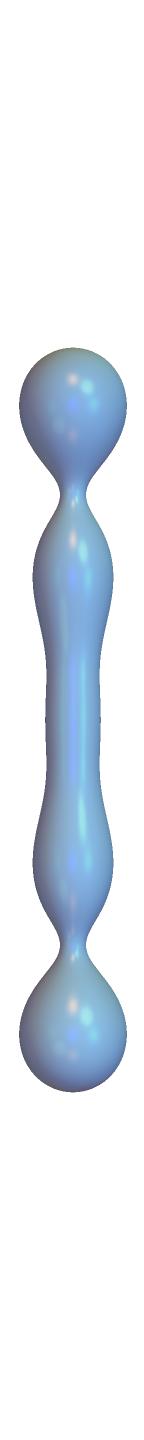}\quad\quad
	\includegraphics[height=12cm]{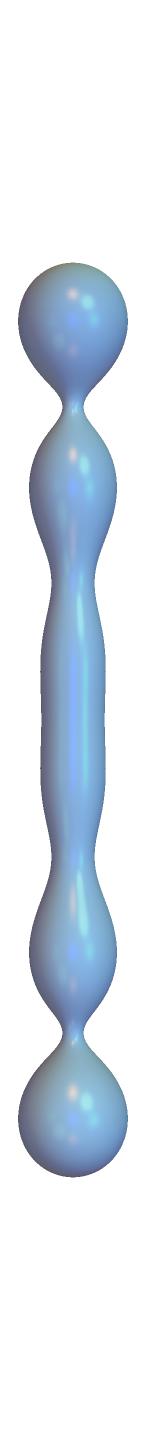}\quad\quad
	\includegraphics[height=12cm]{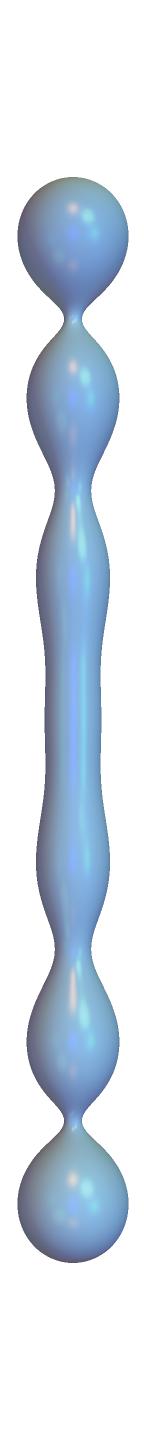}\quad\quad
	\includegraphics[height=12cm]{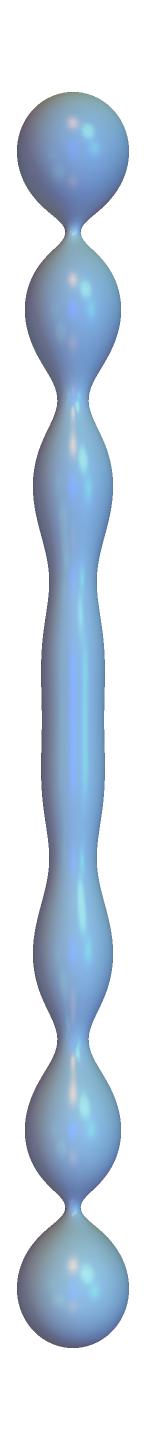}\quad\quad
	\includegraphics[height=12cm]{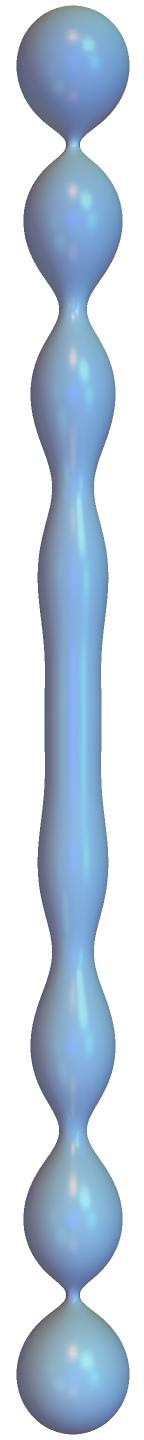}
	\caption{{\small Axially symmetric closed Helfrich surfaces of genus zero with non-constant mean curvature.}}
	\label{Helfrich}
\end{figure}

Figure \ref{Helfrich} reproduces eight axially symmetric closed Helfrich surfaces of genus zero with non-constant mean curvature which, in addition, are symmetric with respect to the plane $\{z=0\}$. The surfaces of Figure \ref{Helfrich} are the first ones that appear as $z_0$ increases from $0$. More examples are given on the final page. We observe that these surfaces can be approximated by gluing together parts of constant mean curvature unduloids. Unduloids depend on two parameters, the constant mean curvature $H_o<0$ and a constant parameter $\tau\in(0,1/\lvert 2H_o\rvert)$. These completely determine an unduloid the following way: if a point $(r,z)$ on the generating curve of an unduloid is related to the
coordinates $(u,v)$ on the generating curve of the sphere via the Gauss map, then
$$2ur+2H_or^2=\tau\,,\quad\quad\quad dz= r_u\:dv\:$$
holds. These equations can be used to parameterize parts of the unduloid via the inverse of its Gauss map. If $H_o<0$ is fixed, the limiting case $\tau=0$ gives a sphere and the case $\tau=1/\lvert 2H_o\rvert$ yields a cylinder. Each fixed surface in the family of axially symmetric Helfrich topological spheres symmetric with respect to $\{z=0\}$ can be approximated by gluing together a sequence of unduloids whose parameter $\tau$ increases from the top part of the surface to the center part. The parameter $H_o$ oscillates  around $-c_o$ in each unduloid of the sequence with very little oscillation. In addition, as $z_0>0$ increases and the surface gets elongated, $H_o$ tends to $-c_o$ and the parameter $\tau$ approaches zero in the top part while it tends to $1/\lvert 2H_o\rvert=1/(2c_o)$ in the center part. This suggests that the top part of the surface approaches spheres with $H=-c_o$ and the center part of the surface approaches a cylinder with mean curvature $H=-c_o$. (This last assertion can also be observed in the center graph on Figure \ref{graphs}.)

In Figure \ref{graphs} we represent the graphs of the functions $\varphi''(\ell)$ (left) and $r_*=r(\ell)$ (center) as the initial height $z_0>0$ varies. To this end, we have fixed several heights $z_0>0$ and solve the system \eqref{system1}-\eqref{system3} for $c_o>0$ with initial conditions \eqref{conditions}. For each solution we have numerically computed the values of $\varphi''$ and $r$ at $z=0$ (i.e., at $s=\ell$). This gave us the red points shown in Figure \ref{graphs}. Then, we have interpolated between these red points to obtain a reliable continuous approximation of the graphs of $\varphi''(\ell)$ and $r(\ell)$ as functions of $z_0$. We observe that the second derivative of the angle $\varphi$ at $z=0$ oscillates around zero as it goes approaching it. This suggests the existence of infinitely many axially symmetric closed Helfrich surfaces of genus zero symmetric with respect to the plane $\{z=0\}$. The first eight zeroes of $\varphi''(\ell)$ correspond to the surfaces of Figure \ref{Helfrich} (the surfaces corresponding to the next eight zeroes can be seen in the last page). Looking at the graph in the center of Figure \ref{graphs} we notice that the radius at which the curve $\gamma(s)=(r(s),z(s))$ solution of \eqref{system1}-\eqref{system3} for the initial conditions \eqref{conditions} meets the $r$-axis oscillates around the value $r=1/(2c_o)$ approaching it (the green straight line). This special value corresponds with the right circular cylinder of radius $1/(2c_o)$ (c.f., Remark \ref{cylinder}). Finally, the planar curve on the right of Figure \ref{graphs} has been constructed by pairing together $\varphi''(\ell)$ and $r(\ell)$. When this curve cuts the vertical axis we obtain $\varphi''(\ell)=0$ and, hence, an axially symmetric closed Helfrich surface of genus zero symmetric with respect to the plane $\{z=0\}$ (see Figure \ref{Helfrich}).

\begin{figure}[h!]
	\centering
	\includegraphics[height=5.4cm]{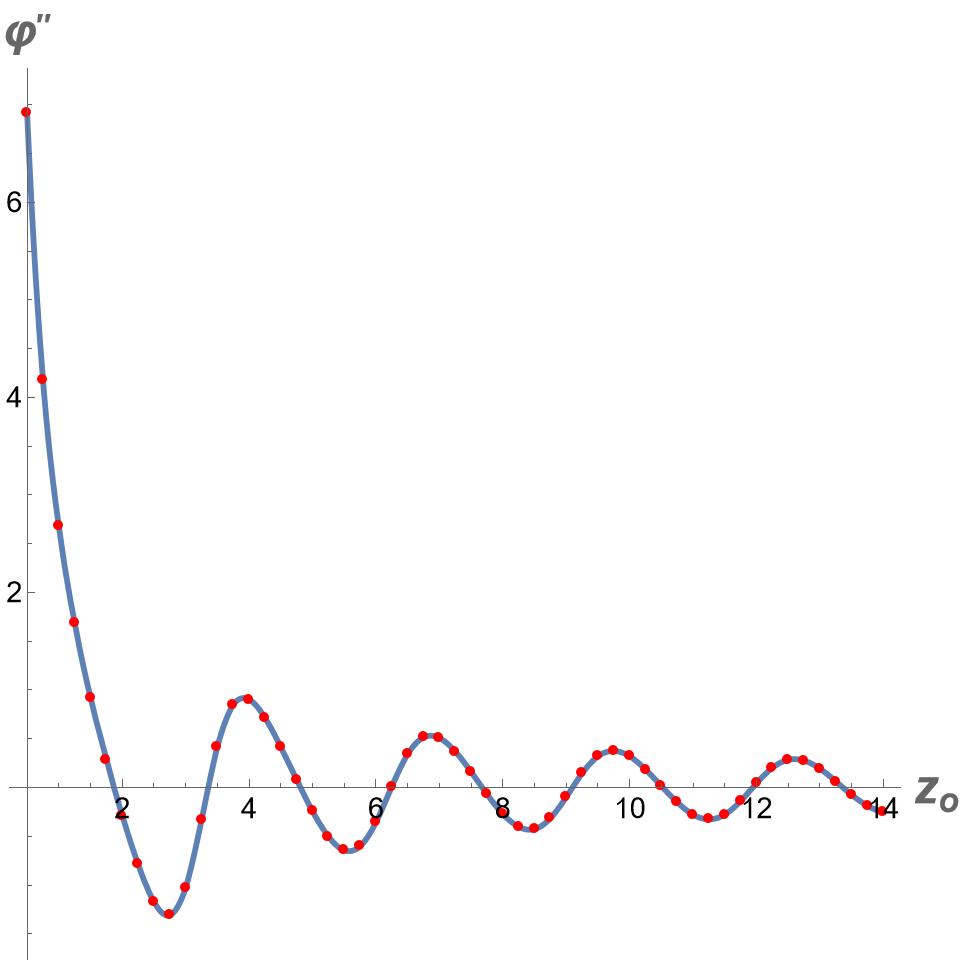}\,\,\,\,
	\includegraphics[height=5.4cm]{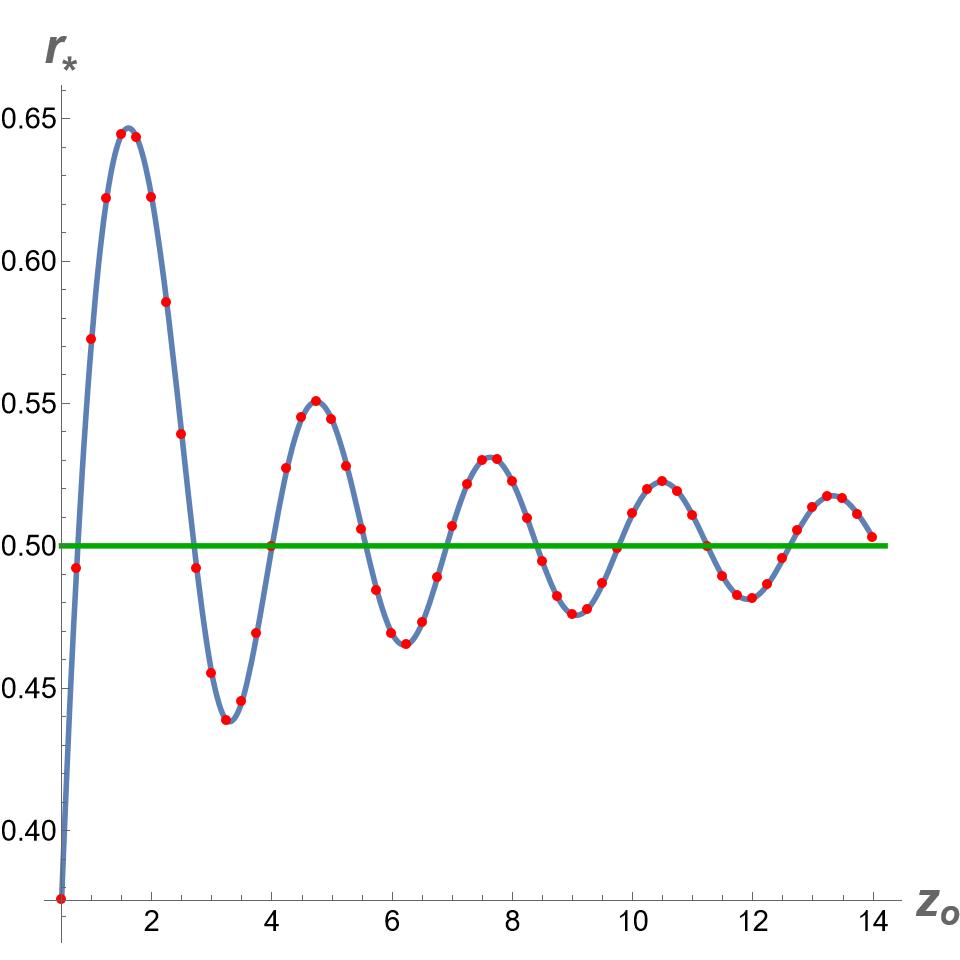}\,\,\,\,
	\includegraphics[height=5.4cm]{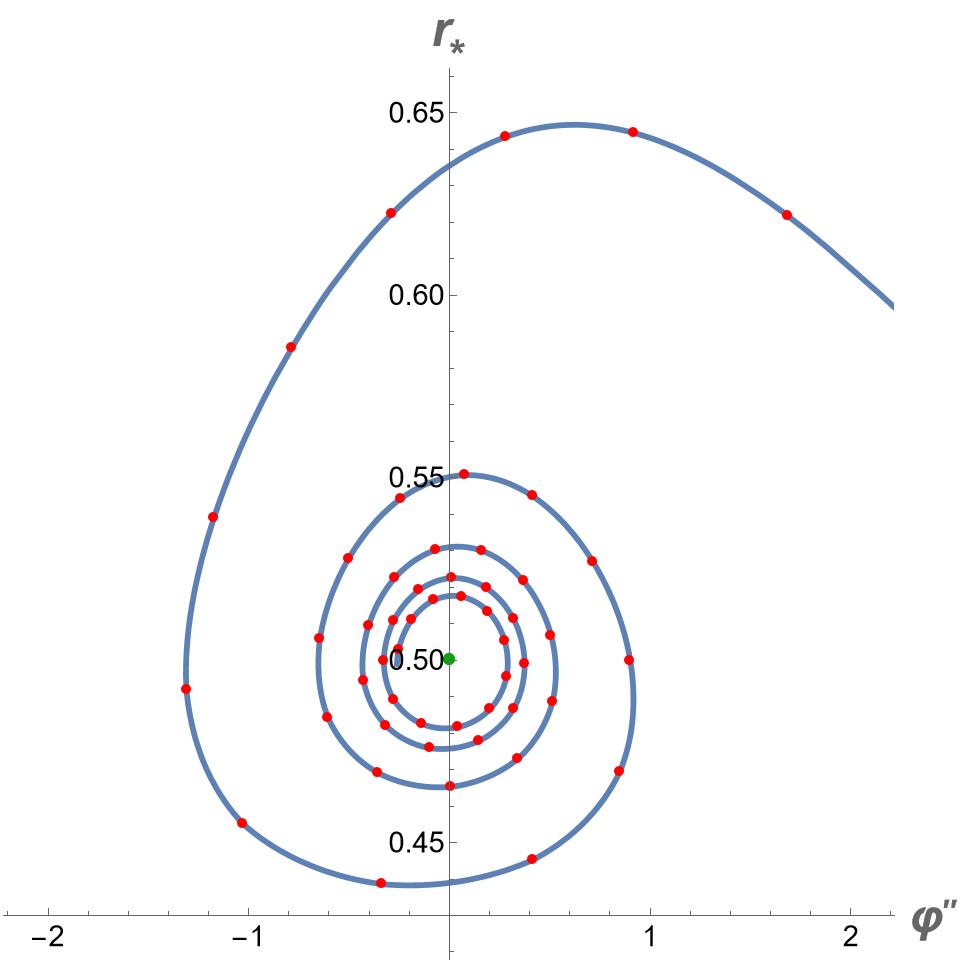}
	\caption{{\small Left: The graph of $\varphi''(\ell)$ as a function of the initial height $z_0>0$. Center: The graph of $r_*=r(\ell)$ as a function of the initial height $z_0>0$. The green horizontal line represents the radius of the critical cylinder, i.e., $r=1/(2c_o)$. Right: A piece of the planar curve $z_0>0\longmapsto (\varphi''(\ell),r_*)$. The green point corresponds to the values of the critical cylinder, i.e., $(0,1/(2c_o))$. For the three cases we have fixed $c_o=1$.}}
	\label{graphs}
\end{figure} 

It is natural to raise the question of whether or not axially symmetric closed Helfrich surfaces of genus zero that are not symmetric with respect to $\{z=0\}$ exist. We recall that a $\mathcal{C}^3$ regularity at $z=0$ is essential (Lemma \ref{2}) and, hence, we must find two solutions $\gamma_+$ and $\gamma_-$ of \eqref{system1}-\eqref{system3} with initial conditions \eqref{conditions} for different initial heights $z_0>0$, such that, $r_+(\ell)=r_-(\ell)$ and $\varphi_+''(\ell)=-\varphi_-''(\ell)$ (see Corollary \ref{H'} and Remark \ref{sec}). These are sufficient and necessary conditions to guarantee that the gluing procedure has sufficient regularity so that the resulting surface is critical for the Helfrich energy $\mathcal{H}$. However, we observe from the representation of the planar curve on the right of Figure \ref{graphs} that these conditions are not compatible. Indeed, by fixing an arbitrary radius $r_*=r(\ell)$, which corresponds with an horizontal line in the figure, we see that the curve $(\varphi''(\ell),r_*)$ may cut this horizontal line several times but never at the same distance with respect to the vertical axis, given the ``spiral-type'' shape of the planar curve. 

This numerical evidence then allows us to propose the following conjecture.

\begin{conjecture}\label{conjecture} Let $\Sigma$ be a closed genus zero surface and $X:\Sigma\longrightarrow\mathbb{R}^3$ an axially symmetric immersion with non-constant mean curvature. The immersion is critical for the Helfrich energy $\mathcal{H}$ if and only if, after a suitable rigid motion and translation of the vertical coordinate, it satisfies the reduced membrane equation \eqref{RME0} and the rescaling conditions \eqref{hh} are satisfied.  
\end{conjecture}

\begin{rem} One of the implications of Conjecture \ref{conjecture} has already been shown in Theorem \ref{rescaling}. Indeed, if the rescaling conditions \eqref{hh} are satisfied, then the rescaling condition \eqref{rc} on the whole surface $\Sigma$ is satisfied automatically. The validity of the converse would imply that Proposition \ref{symmetry} holds for every axially symmetric Helfrich sphere and, hence, these surfaces would necessary be symmetric with respect to the plane $\{z=0\}$. In addition, the proof of the conjecture would lead to the complete classification of axially symmetric Helfrich spheres with non-constant mean curvature. These surfaces would be constructed from the reduced membrane equation \eqref{RME0} looking for those generating curves such that $\varphi''(\ell)=0$ holds. Therefore, the surfaces of Figure \ref{Helfrich} would be the first ones of the family. \end{rem}

\section{Appendix. Circular Biconcave Discoids}

In the literature, there exist examples of non-trivial (other than round spheres) explicit solutions of \eqref{EL} which do not satisfy the reduced membrane equation \eqref{RME0}. For example, the axially symmetric family of circular biconcave discoids  that appeared in \cite{NOO}. These surfaces satisfy the conditions of Theorem \ref{rescaling}, that is, they are axially symmetric topological spheres with non-constant mean curvature. However, they fail to be $\mathcal{C}^2$ at the two cuts with the axis of rotation, a singularity previously observed and which raised questions regarding the suitability of these surfaces to model real red blood cells (see, for instance, \cite{Tu}). We will next show that circular biconcave discoids are not critical points of $\mathcal{H}$. Indeed, the left-hand side of \eqref{EL} for the circular biconcave discoids is a Dirac's delta centered at the two `poles'.

\begin{proposition}\label{discoids}
	Let $\Sigma$ be a closed genus zero surface and $X_D:\Sigma\longrightarrow\r^3$ the axially symmetric immersion obtained by rotating around the $z$-axis the arc length parameterized curve $\gamma(s)=(r(s),z(s))$ where
\begin{equation}\label{fi}
\varphi(s)=\arcsin\left(-2c_o r(s)\log r(s)+A r(s)\right),\quad\quad\quad c_o\neq 0\,,\quad A\in\r\,,
\end{equation}
	is the angle between the positive part of the $r$-axis and the tangent vector to $\gamma(s)$. Then, $\delta \mathcal{H}[\Sigma]\neq 0$. In other words, the immersion $X_D$ is not critical for the Helfrich energy $\mathcal{H}$.
\end{proposition}
\begin{proof} Let $\gamma(s)=(r(s),z(s))$ be the arc length parameterized generating curve of the circular biconcave discoid $X_D(\Sigma)$. We will assume that $s$ is measured from the north `pole'. From the definition of the angle $\varphi(s)$ it follows that $r'(s)=\cos\varphi(s)$ and $z'(s)=\sin\varphi(s)$, where $\left(\,\right)'$ denotes the derivative with respect to the arc length parameter $s$.

Considering the unit normal $\nu$ pointing outward of the closed surface $X_D(\Sigma)$, the mean curvature $H$ is given by \eqref{H}. Combining this with \eqref{fi}, we obtain that
\begin{equation}\label{HDiscoid}
	H+c_o=-2c_o\log r+A\,.
\end{equation}
It is then a straightforward computation to check that \eqref{EL} is satisfied for $r>0$. However, as we will see next, at $r=0$ the immersion $X_D$ is of class $\mathcal{C}^1$ but not $\mathcal{C}^2$ and, hence, \eqref{EL} does not characterize equilibria for $\mathcal{H}$ there.

In order to check the regularity at $r=0$ of the immersion $X_D$ it is enough to analyze how the generating curve $\gamma(s)=(r(s),z(s))$ meets the $z$-axis. From \eqref{fi}, we deduce that when $r\to 0$, $\varphi\to 0$ and, thus, the curve $\gamma$ cuts the $z$-axis orthogonally. This proves that $X_D$ is of class $\mathcal{C}^1$ at $r=0$. Next, differentiating twice $\gamma(s)$ with respect to the arc length parameter $s$ and using the relations between $r'(s)$ and $z'(s)$ with the angle $\varphi(s)$, we get $\gamma''(s)=(-\varphi'(s)\sin\varphi(
s),\varphi'(s)\cos\varphi(s))$. The proof that $X_D$ is not of class $\mathcal{C}^2$ at $r=0$ follows by verifying that these second derivatives do not exist when $r=0$. Indeed, it is enough to compute
$$\varphi'(s)\cos\varphi(s)=\left(\sin\varphi(s)\right)'=\left(-2c_o\log r(s)-2c_o+A\right)\cos\varphi(s)\longrightarrow\pm\infty\,,$$
when $r\to 0$.

We will now compute the first variation of the Helfrich energy $\mathcal{H}$ at the immersion $X_D$. Due to the lack of regularity, the equation \eqref{EL} does not characterize the critical points when $r=0$ and, hence, we must remove those singularities and take the corresponding limit. Denote by $D_T$ and $D_B$ the top (respectively, bottom) disc type graphs in $\Sigma$ around $r=0$ of radius $\epsilon>0$ sufficiently small and let $\Sigma_\epsilon=\Sigma\setminus\left( D_T\cup D_B\right)$. From the first variation formula of the Helfrich energy  we have, for normal variations $\delta X=\psi\nu$, $\psi\in\mathcal{C}_o^\infty(\Sigma)$,
\begin{eqnarray*}
	\delta\left(\int_{\Sigma_\epsilon}(H+c_o)^2\,d\Sigma\right)&=&\int_{\Sigma_\epsilon}\left(\Delta H+2(H+c_o)\left(H(H-c_o)-K\right)\right)\psi\,d\Sigma\\&&+\oint_{\partial\Sigma_\epsilon}\left((H+c_o)\partial_n\psi-\partial_n H\psi\right)d\sigma\,,
\end{eqnarray*}
where $\partial_n$ is the derivative in the conormal direction. Since the immersion $X_D$ satisfies \eqref{EL} for $r>0$, the integrand of the surface integral above is zero. On the other hand, the boundary integral can be decomposed as the integral over $\partial D_T$ and $\partial D_B$, with the suitable choice of orientation. 

The immersion $X_D$ is symmetric with respect to the plane $\{z=0\}$, so we will only consider the top disc $D_T$ (the argument in $D_B$ is similar). On the graph $D_T$, the derivative in the conormal direction can be computed as
$$\partial_n=\frac{1}{\sqrt{1+z_r^2\,}}\,\partial_r\,.$$
Hence, using \eqref{HDiscoid} and the expression of the conormal derivative in terms of the radial derivative, we obtain
$$\oint_{\partial D_T} (H+c_o)\partial_n\psi\,d\sigma=\oint_{\partial D_T} \left(2c_o\log r+A\right)\partial_n\psi\,d\sigma=\int_0^{2\pi}\frac{\epsilon \left(2c_o \log\epsilon +A\right)\partial_r\psi}{\sqrt{1+z_r^2\,}}\,d\theta\longrightarrow 0\,,$$
when $\epsilon\to 0$. While, for the other integral,
$$\oint_{\partial D_T} \partial_nH\psi\,d\sigma=-\oint_{\partial D_T} 2c_o\partial_n\left(\log r\right)\psi\,d\sigma=-\int_0^{2\pi}\frac{2c_o\psi}{\sqrt{1+z_r^2\,}}\,d\theta\longrightarrow -4\pi c_o\psi(0)\,,$$
when $\epsilon\to 0$. Observe that $\psi(0)$ is, in general, different from zero and, since $c_o\neq 0$, so is this limit.

Combining above computations on $D_T$ with the analogue ones in $D_B$ (omitted here for the sake of simplicity), we conclude with
$$\delta\mathcal{H}[\Sigma]=\lim_{\epsilon\to 0} \delta\left(\int_{\Sigma_\epsilon}(H+c_o)^2\,d\Sigma\right)=8\pi c_o\psi(0)\,.$$
Consequently, the first variation is not always zero and, hence, $X_D$ is not critical for the Helfrich energy $\mathcal{H}$. Observe also that, but for the coefficient $8\pi c_o$, this computation shows that the left-hand side of \eqref{EL} at $X$ is a Dirac's delta centered at $r=0$. \end{proof}

\begin{figure}[h!]
	\centering
	\includegraphics[height=4.1cm]{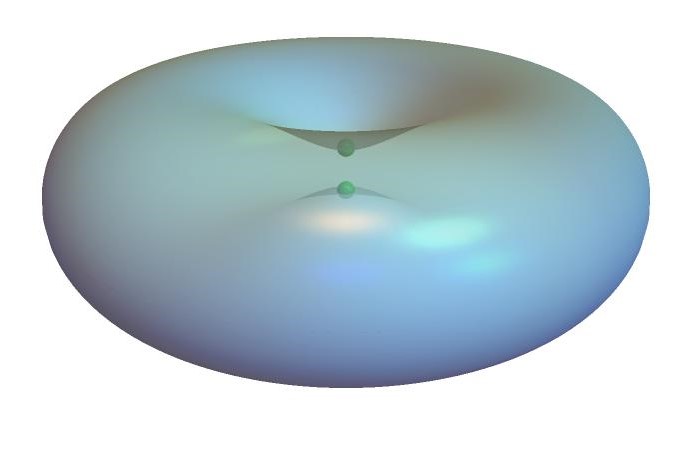}\,
	\includegraphics[height=4.1cm]{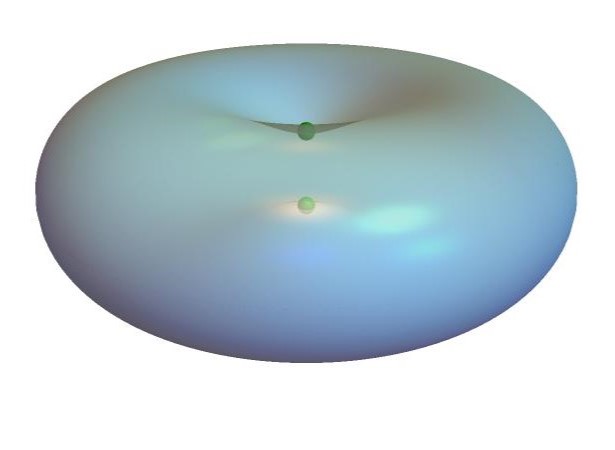}\,
	\includegraphics[height=4.1cm]{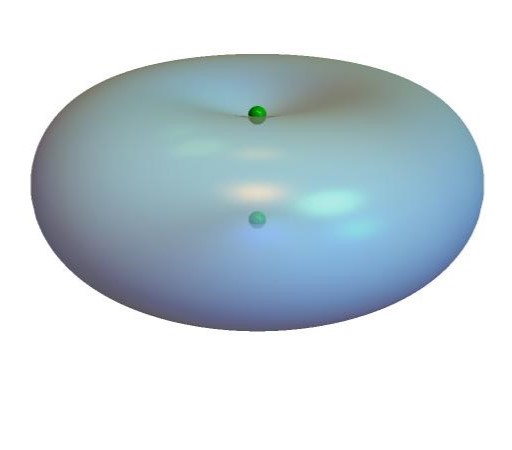}
	\caption{{\small Three circular biconcave discoids, that is, axially symmetric topological spheres with non-constant mean curvature solution of \eqref{EL} for $r>0$. The green points are the `poles', where the surfaces are only of class $\mathcal{C}^1$. Due to this lack of regularity, these discoids are not Helfrich surfaces.}}
	\label{Discoids}
\end{figure}

\begin{rem} Although circular biconcave discoids are not critical points of the Helfrich energy $\mathcal{H}$, they are central in the theory of Helfrich surfaces. Along these surfaces we have $H^2-K=c_o^2$  and, hence, they play a similar role to totally umbilical spheres for the Willmore energy. In the sense that the modified conformal Gauss map $Y^{c_o}$  fails to be a definite immersion at all the points of these surfaces \cite{PP2}.
\end{rem}

\section*{Acknowledgments}

The authors would like to thank the referees for carefully reviewing the paper.

Rafael L\'opez  has been partially supported by MINECO/MICINN/FEDER grant no. PID2023-150727NB-I00,  and by the ``Mar\'{\i}a de Maeztu'' Excellence Unit IMAG, reference CEX2020-001105- M, funded by MCINN/AEI/10.13039/ 501100011033/ CEX2020-001105-M.



\begin{flushleft}
	Rafael L{\footnotesize \'OPEZ}\\
	Departamento de Geometr\'ia y Topolog\'ia, Universidad de Granada, 18071 Granada, Spain\\
	E-mail: rcamino@ugr.es
\end{flushleft}

\begin{flushleft}
	Bennett P{\footnotesize ALMER}\\
	Department of Mathematics and Statistics, Idaho State University, Pocatello, ID, 83209, USA\\
	E-mail: palmbenn@isu.edu
\end{flushleft}

\begin{flushleft}
	\'Alvaro P{\footnotesize \'AMPANO}\\
	Department of Mathematics and Statistics, Texas Tech University, Lubbock, TX, 79409, USA\\
	E-mail: alvaro.pampano@ttu.edu
\end{flushleft}

\newpage

\begin{figure}[h!]
	\centering
	\includegraphics[height=24cm]{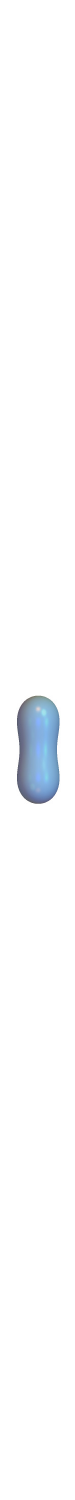}\quad\quad
	\includegraphics[height=24cm]{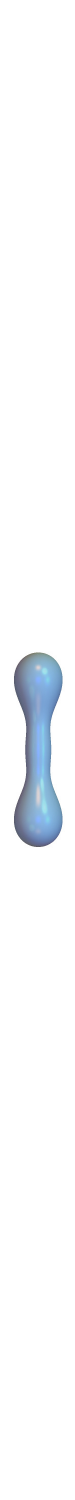}\quad\quad
	\includegraphics[height=24cm]{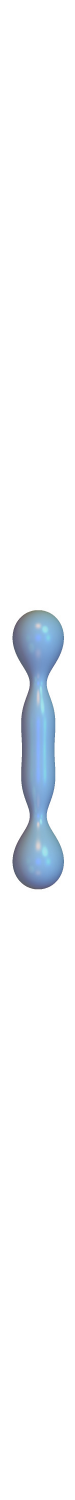}\quad\quad
	\includegraphics[height=24cm]{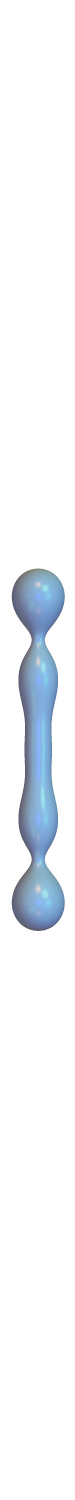}\quad\quad
	\includegraphics[height=24cm]{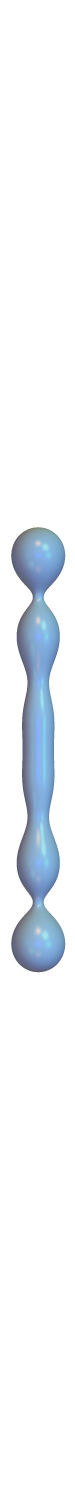}\quad\quad
	\includegraphics[height=24cm]{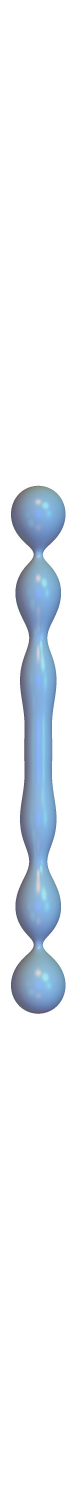}\quad\quad
	\includegraphics[height=24cm]{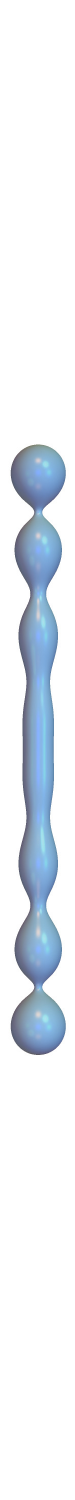}\quad\quad
	\includegraphics[height=24cm]{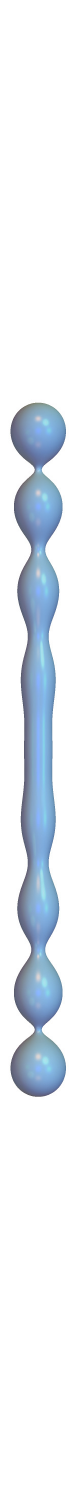}
\end{figure}

\begin{figure}[h!]
	\centering
	\includegraphics[height=24cm]{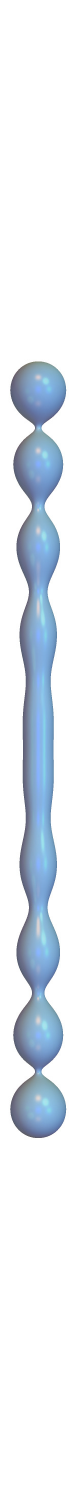}\quad\quad
	\includegraphics[height=24cm]{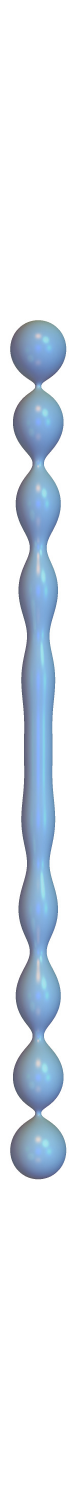}\quad\quad
	\includegraphics[height=24cm]{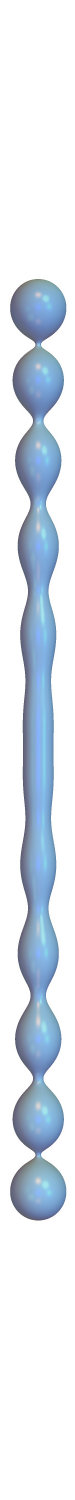}\quad\quad
	\includegraphics[height=24cm]{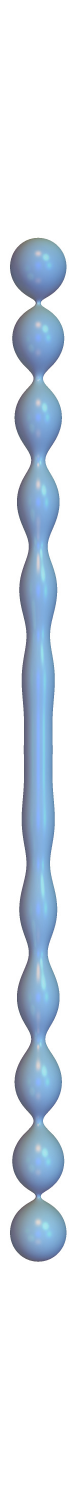}\quad\quad
	\includegraphics[height=24cm]{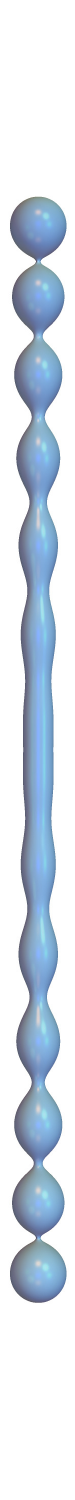}\quad\quad
	\includegraphics[height=24cm]{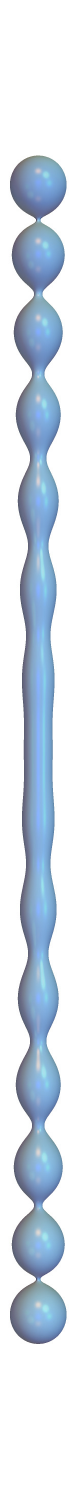}\quad\quad
	\includegraphics[height=24cm]{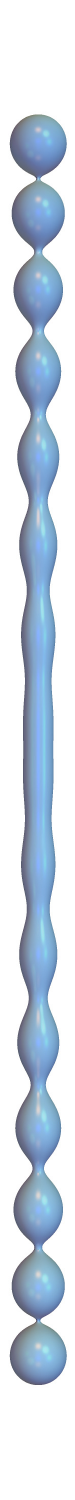}\quad\quad
	\includegraphics[height=24cm]{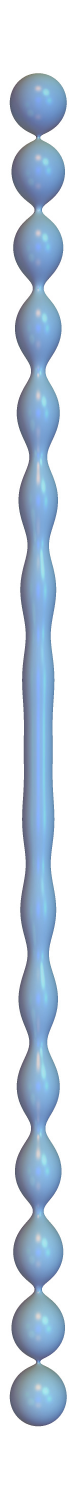}
\end{figure}

\end{document}